\pdfoutput=1
\RequirePackage{ifpdf}
\ifpdf 
\documentclass[pdftex]{sigma}
\else
\documentclass{sigma}
\fi

\usepackage{enumerate,eucal}
\usepackage{tikz}
\usetikzlibrary{positioning}

\numberwithin{equation}{section}

\newtheorem{Theorem}{Theorem}[section]
\newtheorem{Corollary}[Theorem]{Corollary}
\newtheorem{Lemma}[Theorem]{Lemma}
\newtheorem{Proposition}[Theorem]{Proposition}
{ \theoremstyle{definition}
\newtheorem{Definition}[Theorem]{Definition}
\newtheorem{Assumption}[Theorem]{Assumption}
\newtheorem{Remark}[Theorem]{Remark}}

\newcommand\scalemath[2]{\scalebox{#1}{\mbox{\ensuremath{\displaystyle #2}}}}

\newcommand{\gauss}[2]{\genfrac{[}{]}{0pt}{}{#1}{#2}}

\usepackage{scalerel}
\newsavebox{\foobox}
\newcommand{\slantbox}[2][0]{\mbox{%
 \sbox{\foobox}{#2}%
 \hskip\wd\foobox
 \pdfsave
 \pdfsetmatrix{1 0 #1 1}%
 \llap{\usebox{\foobox}}%
 \pdfrestore
}}
\newcommand\unslant[2][-.2]{%
 \mkern.5mu%
 \ThisStyle{\slantbox[#1]{$\SavedStyle#2$}}%
 \mkern-1mu%
}

\DeclareSymbolFont{sfletters}{OML}{iwonal}{m}{it}

\DeclareMathSymbol{\stheta}{\mathord}{sfletters}{"12}
\DeclareMathSymbol{\sphi}{\mathord}{sfletters}{"1E}
\DeclareMathSymbol{\svarphi}{\mathord}{sfletters}{"27}
\DeclareMathSymbol{\salpha}{\mathord}{sfletters}{"0B}
\DeclareMathSymbol{\sbeta}{\mathord}{sfletters}{"0C}
\DeclareMathSymbol{\sgamma}{\mathord}{sfletters}{"0D}
\DeclareMathSymbol{\stau}{\mathord}{sfletters}{"1C}
\newcommand{\mytheta}{\unslant\stheta}
\newcommand{\myphi}{\unslant\sphi}
\newcommand{\myvarphi}{\unslant\svarphi}
\newcommand{\myalpha}{\unslant\salpha}
\newcommand{\mybeta}{\unslant\sbeta}
\newcommand{\mygamma}{\unslant\sgamma}
\newcommand{\mytau}{\unslant\stau}

\DeclareSymbolFont{bsfletters}{OT1}{cmsmf}{bx}{n}
\DeclareMathSymbol{\Wsf}{\mathord}{bsfletters}{"57}

\begin{document}

\allowdisplaybreaks

\newcommand{\arXivNumber}{1709.07825}

\renewcommand{\PaperNumber}{009}

\FirstPageHeading

\ShortArticleName{Non-Symmetric Dual $q$-Krawtchouk Polynomials}

\ArticleName{Dual Polar Graphs, a nil-DAHA of Rank One, \\ and Non-Symmetric Dual $\boldsymbol{q}$-Krawtchouk Polynomials}

\Author{Jae-Ho LEE~$^\dag$ and Hajime TANAKA~$^\ddag$}

\AuthorNameForHeading{J.-H.~Lee and H.~Tanaka}

\Address{$^\dag$~Department of Mathematics and Statistics, University of North Florida,\\
\hphantom{$^\dag$}~Jacksonville, FL 32224, USA}
\EmailD{\href{mailto:jaeho.lee@unf.edu}{jaeho.lee@unf.edu}}
\URLaddressD{\url{https://sites.google.com/site/jaeholeemath/}}

\Address{$^\ddag$~Research Center for Pure and Applied Mathematics, Graduate School of Information Sciences,\\
\hphantom{$^\ddag$}~Tohoku University, Sendai 980-8579, Japan}
\EmailD{\href{mailto:htanaka@tohoku.ac.jp}{htanaka@tohoku.ac.jp}}
\URLaddressD{\url{http://www.math.is.tohoku.ac.jp/~htanaka/}}

\ArticleDates{Received September 25, 2017, in f\/inal form January 29, 2018; Published online February 10, 2018}

\Abstract{Let $\Gamma$ be a dual polar graph with diameter $D \geqslant 3$, having as vertices the maximal isotropic subspaces of a~f\/inite-dimensional vector space over the f\/inite f\/ield $\mathbb{F}_q$ equipped with a non-degenerate form (alternating, quadratic, or Hermitian) with Witt index~$D$. From a~pair of a~vertex~$x$ of $\Gamma$ and a maximal clique~$C$ containing~$x$, we construct a~$2D$-dimensional irreducible module for a nil-DAHA of type $(C^{\vee}_1, C_1)$, and establish its connection to the generalized Terwilliger algebra with respect to~$x$,~$C$. Using this module, we then def\/ine the non-symmetric dual $q$-Krawtchouk polynomials and derive their recurrence and orthogonality relations from the combinatorial points of view. We note that our results do not depend essentially on the particular choice of the pair~$x$,~$C$, and that all the formulas are described in terms of~$q$,~$D$, and one other scalar which we assign to $\Gamma$ based on the type of the form.}


\Classification{05E30; 20C08; 33D45; 33D80}

\section{Introduction}

$Q$-\emph{polynomial distance-regular graphs} are viewed as f\/inite analogues of compact symmetric spaces of rank one, and have been extensively studied; cf.~\cite{BI,BCN,DKT,CG}. By a famous theorem of Leonard~\cite{Leonard1982SIAM}, \cite[Section~3.5]{BI}, the duality property of $Q$-polynomial distance-regular graphs characterizes the terminating branch of the Askey scheme~\cite{KLS2010B} of (basic) hypergeometric orthogonal polynomials, at the top (i.e., $_4\phi_3$) of which are the $q$-Racah polynomials. A central tool in studying such a graph is the \emph{Terwilliger algebra} $T=T(x)$ \cite{PT.1992(1),PT.1993(1),PT.1993(2)}, which is a~non-commutative semisimple matrix $\mathbb{C}$-algebra attached to every vertex $x$ of the graph.

Cherednik \cite{IC,Cherednik1995AM,Cherednik1995IM,Cherednik1995IMRN} introduced the \emph{double affine Hecke algebras} (DAHAs) for reduced af\/f\/ine root systems and used them to prove several conjectures for the Macdonald polynomials. Sahi~\cite{SS} then extended the concept to the non-reduced af\/f\/ine root systems of type $(C^{\vee}_n, C_n)$ and proved the duality conjecture and other conjectures for the Koornwinder polynomials, which are the Macdonald polynomials attached to the af\/f\/ine root systems of type $(C^{\vee}_n, C_n)$. For $n=1$, these polynomials are the Askey--Wilson polynomials which are of $_4\phi_3$, and the $q$-Racah polynomials are a discretization of the Askey--Wilson polynomials.

Recently, the f\/irst author \cite{JHL} found a link between the theories of $Q$-polynomial distance-regular graphs and the DAHAs. Namely, he considered a $Q$-polynomial distance-regular graph~$\Gamma$ corresponding to $q$-Racah polynomials. He further assumed that $\Gamma$ possesses a clique $C$ with maximal possible size (called a \emph{Delsarte clique}), and def\/ined a semisimple matrix $\mathbb{C}$-algebra \smash{$\mathbf{T}=\mathbf{T}(x,C)$} attached to~$C$ \emph{and} a vertex $x\in C$, which contains $T(x)$ as a subalgebra. Then he showed that the so-called \emph{primary} $\mathbf{T}$-module has the structure of an irreducible module for the DAHA of type $(C^{\vee}_1, C_1)$, and studied how the two module structures are related. In the subsequent paper~\cite{JHL2}, he captured in this context what should be called the \emph{non-symmetric} $q$-\emph{Racah polynomials}, which are the f\/inite counterpart of the non-symmetric Askey--Wilson polynomials discussed by Sahi~\cite{SS}, and succeeded in giving an explicit combinatorial description of their orthogonality relations.
We note that the expression for the non-symmetric $q$-Racah polynomials in~\cite{JHL2} agrees with the one for the non-symmetric Askey--Wilson polynomials given by Koornwinder and Bouzef\/four~\cite{KB2011AA}.

A big goal in this project is to establish a ``non-symmetric version'' of Leonard's theorem mentioned above. As the next attempt towards this goal, we discuss the \emph{dual polar graphs} in this paper, and specialize the above situation to this case. The dual polar graphs are a classical family of $Q$-polynomial distance-regular graphs arising naturally as homogeneous spaces of f\/inite classical groups by maximal parabolic subgroups, and correspond to the dual $q$-Krawtchouk polynomials which are of~$_3\phi_2$; cf.~\cite{Stanton1980AJM,Stanton1981GD}. In particular, we will obtain the \emph{non-symmetric} dual $q$-Krawtchouk polynomials and describe their recurrence and orthogonality relations; cf.\ Theorems~\ref{thm4} and~\ref{thm3}. There are multiple motivations for the present work. First, for the $q$-Racah case, there is indeed no known example of a $Q$-polynomial distance-regular graph with large diameter (say, at least ten) having such a maximal clique, except the ordinary polygons with even order, which we view as rather trivial.
Hence we may say that the theory developed in~\cite{JHL,JHL2} still remains somewhat at the algebraic/parametric level, whereas we will deal with concrete (and non-trivial) combinatorial examples in this paper. Second, there are of course other candidates of examples to be considered, such as the Grassmann graphs corresponding to the dual $q$-Hahn polynomials which lie in between the $q$-Racah and the dual $q$-Krawtchouk polynomials, but we decided to focus on the dual polar graphs, mainly because they exhibit quite a strong regularity of being \emph{regular near polygons}, so that the computations become far simpler than those in~\cite{JHL,JHL2}. Though many of our results can also be obtained in principle by taking appropriate limits of the (much involved) results in~\cite{JHL,JHL2}, this fact motivates us to work out the details for this particular case rather independently of~\cite{JHL,JHL2}. Third, we will encounter a~\emph{nil-DAHA} of type $(C^{\vee}_1, C_1)$, which is obtained by specializing some of the def\/ining relations of the DAHA of type $(C^{\vee}_1, C_1)$.
The nil-DAHAs were introduced and discussed recently by Cherednik and Orr~\cite{CO1,CO2,CO3}, and our results demonstrate the fundamental importance of the concept in the theory of $Q$-polynomial distance-regular graphs; cf.~Theorems~\ref{thm1} and~\ref{thm2}. We may remark that our results are also relevant to recent work by Mazzocco~\cite{Mazzocco2014SIGMA, Mazzocco2016N}; cf.~Remarks~\ref{Mazzocco's work 1} and~\ref{Mazzocco's work 2}.

The layout of the paper is as follows. In Section~\ref{S:DRG}, we recall some background concerning $Q$-polynomial distance-regular graphs and their Terwilliger algebras. In Section~\ref{S:DPG}, we recall the dual polar graphs and their properties. We f\/ix a dual polar graph $\Gamma$ with diameter $D \geqslant 3$, having as vertices the maximal isotropic subspaces of a f\/inite-dimensional vector space over the f\/inite f\/ield $\mathbb{F}_q$ equipped with a non-degenerate form (alternating, quadratic, or Hermitian) with Witt index~$D$. We f\/ix a maximal clique $C$ of~$\Gamma$, and discuss the Terwilliger algebra of $\Gamma$ \emph{with respect to} $C$ in the sense of Suzuki~\cite{HS}. In Section~\ref{S:algT}, we also f\/ix a vertex $x \in C$ and def\/ine the \emph{generalized Terwilliger algebra} $\mathbf{T}=\mathbf{T}(x,C)$ of~$\Gamma$. We then def\/ine from $x$ and $C$ a $2D$-dimensional subspace~$\mathbf{W}$ of the standard module for $\mathbf{T}$, and show that it is an irreducible $\mathbf{T}$-module. In Section~\ref{S:pre LS}, we recall the basic theory concerning \emph{Leonard systems}, a certain linear algebraic framework for Leonard's theorem introduced by Terwilliger~\cite{PT.2001}. We focus on the class of Leonard systems corresponding to the dual $q$-Krawtchouk polynomials and discuss their properties. To the $\mathbf{T}$-module $\mathbf{W}$ we attach four Leonard systems, all of which belong to this class, and we study each of these Leonard systems in detail in Sections~\ref{S:T-mod W} and~\ref{S:T(C)-mod}. In Section~\ref{S:nil-DAHA}, we introduce a nil-DAHA $\overline{\mathcal{H}}$ of type $(C^{\vee}_1, C_1)$, and def\/ine a $2D$-dimensional representation $\overline{\mathcal{H}}\rightarrow\operatorname{End}(\mathbf{W})$. Our f\/irst main results of this paper describe how the $\mathbf{T}$-action on~$\mathbf{W}$ is related to the $\overline{\mathcal{H}}$-action; cf.~Theorems~\ref{thm1} and~\ref{thm2}. In Section~\ref{S:nsdqKpoly}, we def\/ine the non-symmetric dual $q$-Krawtchouk polynomials $\ell^{\pm}_i$ using the representation $\overline{\mathcal{H}}\rightarrow\operatorname{End}(\mathbf{W})$ and discuss a role of these Laurent polynomials in~$\mathbf{W}$. In fact, we will obtain two expressions for the~$\ell_i^{\pm}$. Recurrence relations involving at most four terms will also be given; cf.~Theorem~\ref{thm4}. The standard (Hermitian) inner product on $\mathbf{W}$ gives rise to an inner product on the $2D$-dimensional vector space to which the $\ell^{\pm}_i$ belong, which ultimately leads in Section~\ref{S:orth rels} to our second main result of this paper, i.e., a~combinatorial description of the orthogonality relations for the~$\ell^{\pm}_i$; cf.~Theorem \ref{thm3}. We note that our results do not depend essentially on the particular choice of the pair~$x$,~$C$, and that all the formulas are described in terms of~$q$,~$D$, and one other scalar~$e$ which we assign to $\Gamma$ based on the type of the non-degenerate form.

Throughout this paper, we use the following notation. For a given non-empty f\/inite set $X$, let $\operatorname{Mat}_X(\mathbb{C})$ be the $\mathbb{C}$-algebra consisting of the complex square matrices indexed by~$X$. Let $V=V_X$ be the $\mathbb{C}$-vector space consisting of the complex column vectors indexed by~$X$. We endow $V$ with the inner product $\langle u,v\rangle =u^t\overline{v}$ for $u,v \in V$, where $^t$ denotes transpose and $\bar{ \ }$ denotes complex conjugate. We abbreviate $\|u\|^2 = \langle u,u\rangle$ for all $u \in V$. For every $y\in X$, let~$\hat{y}$ be the vector in~$V$ with a~$1$ in the $y$-coordinate and~$0$ elsewhere. For a subset $Y\subset X$, let~$\hat{Y}=\sum\limits_{y\in Y}\hat{y} \in V$ denote its (column) characteristic vector. A Laurent polynomial $f(\eta)\in \mathbb{C}[\eta,\eta^{-1}]$ in the variable $\eta$ is said to be \emph{symmetric} if $f(\eta)=f(\eta^{-1})$, and \emph{non-symmetric} otherwise. Note that the symmetric Laurent polynomials are precisely the polynomials in $\xi:=\eta+\eta^{-1}$. Let $q$ be a prime power. For $r\in \mathbb{C}$ and an integer $n \geqslant 0$, let
\begin{gather*}
	(r;q)_n = (1-r)(1-rq)\cdots\big(1-rq^{n-1}\big), \qquad \gauss{n}{1}=\gauss{n}{1}_{\!q}= \frac{q^n-1}{q-1}.
\end{gather*}
For $r_1, r_2, \dots, r_{k+1}, s_1, s_2, \dots, s_k \in \mathbb{C}$, let
\begin{align*}
	& {}_{k+1} \phi_k
	\left( \left.
	\begin{matrix}
	r_1, r_2, \dots, r_{k+1} \\
	s_1, s_2, \dots, s_k
	\end{matrix}
	\right| q,\eta \right)= \sum^{\infty}_{n=0} \frac{(r_1;q)_n (r_2;q)_n\cdots (r_{k+1};q)_n}{(s_1;q)_n (s_2;q)_n \cdots (s_k;q)_n}\frac{\eta^n}{(q;q)_n}.
\end{align*}

\section{Distance-regular graphs}\label{S:DRG}

Let $\Gamma$ be a f\/inite simple connected graph with vertex set $X$ and diameter $D$. For $x\in X$, let
\begin{gather*}
	\Gamma_i(x) = \{ y \in X \colon \partial(x,y)=i \}, \qquad 0 \leqslant i \leqslant D,
\end{gather*}
where $\partial$ denotes the path-length distance. We abbreviate $\Gamma(x):=\Gamma_1(x)$. We call $\Gamma$ \emph{distance-regular} if there are non-negative integers $a_i$, $b_i$, $c_i$, $0\leqslant i\leqslant D$, called the \emph{intersection numbers} of $\Gamma$, such that $b_D=c_0=0$, $b_{i-1}c_i\ne 0$, $1\leqslant i\leqslant D$, and
\begin{gather*}
	a_i=|\Gamma_i(x)\cap\Gamma(y)|, \qquad	b_i=|\Gamma_{i+1}(x)\cap\Gamma(y)|, \qquad	c_i=|\Gamma_{i-1}(x)\cap\Gamma(y)|
\end{gather*}
for every pair of vertices $x,y \in X$ with $\partial(x,y)=i$, where $\Gamma_{-1}(x)=\Gamma_{D+1}(x):=\varnothing$. Fig.~\ref{2x4-grid} shows a small example of a distance-regular graph with $D=3$.

\begin{figure}
\centering
\begin{tikzpicture}
[bit/.style={circle,inner sep=0mm,minimum size=1.3mm,draw=black}]
\node [bit] (1) {};
\node [bit] (2) [right=of 1] {};
\node [bit] (3) [right=of 2] {};
\node [bit] (4) [right=of 3] {};
\node [bit] (5) [below=of 1,yshift=-3mm] {};
\node [bit] (6) [right=of 5] {};
\node [bit] (7) [right=of 6] {};
\node [bit] (8) [right=of 7] {};
\draw (1.south east) -- (6.north west);
\draw (1.south east) -- (7.north west);
\draw (1.south east) -- (8.north west);
\draw (2.south west) -- (5.north east);
\draw (2.south east) -- (7.north west);
\draw (2.south east) -- (8.north west);
\draw (3.south west) -- (5.north east);
\draw (3.south west) -- (6.north east);
\draw (3.south east) -- (8.north west);
\draw (4.south west) -- (5.north east);
\draw (4.south west) -- (6.north east);
\draw (4.south west) -- (7.north east);
\end{tikzpicture}
\caption{The complement of the $2\times 4$-grid.}\label{2x4-grid}
\end{figure}
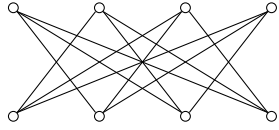

From now on, assume that $\Gamma$ is distance-regular. The $i^{\mathrm{th}}$ \emph{distance matrix} of $\Gamma$ is the $0$-$1$ matrix $A_i \in \operatorname{Mat}_X(\mathbb{C})$ such that $(A_i)_{x,y}=1$ if and only if $\partial(x,y)=i$. The \emph{Bose--Mesner algebra} of $\Gamma$ is the subalgebra $M$ of $\operatorname{Mat}_X(\mathbb{C})$ generated by the $A_i$. We note that $M$ is semisimple since it is closed under conjugate-transposition. Observe also that
\begin{gather*}
	A_1A_i=b_{i-1}A_{i-1}+a_iA_i+c_{i+1}A_{i+1}, \qquad 0\leqslant i\leqslant D,
\end{gather*}
from which it follows that, for $0 \leqslant i \leqslant D$, there is a polynomial $v_i \in \mathbb{C}[\xi]$ with $\deg (v_i)=i$ such that $v_i(A_1)=A_i$. It follows that the adjacency matrix $A:=A_1$ of $\Gamma$ generates $M$, and that the~$A_i$ form a basis for $M$. In particular, we have $\dim(M)=D+1$.

Since $A$ is real symmetric and generates $M$, it has $D+1$ mutually distinct real eigenvalues $\theta_0, \theta_1, \dots, \theta_D$, which we call the \emph{eigenvalues of} $\Gamma$.
We will always set $\theta_0:=b_0$, the valency (or degree) of $\Gamma$.
For $0 \leqslant i \leqslant D$, let $E_i \in \operatorname{Mat}_X(\mathbb{C})$ be the orthogonal projection
onto the eigenspace of $\theta_i$.
Then we have
\begin{gather*}
	A=\sum_{i=0}^D\theta_i E_i,
\end{gather*}
and the $E_i$ form another basis for $M$ consisting of the primitive idempotents, i.e., $E_iE_j=\delta_{i,j}E_i$, $\sum\limits_{i=0}^DE_i=I$ (the identity matrix). We note that $E_0V=\mathbb{C}\hat{X}$. Observe that~$M$ is also closed under entrywise multiplication, denoted $\circ$.
We say that $\Gamma$ is $Q$-\emph{polynomial} with respect to the ordering $\{E_i\}_{i=0}^D$ (or $\{\theta_i\}_{i=0}^D$) if there are scalars $a_i^*$, $b_i^*$, $c_i^*$, $0\leqslant i\leqslant D$, such that $b_D^*=c_0^*=0$, $b_{i-1}^*c_i^*\ne 0$, $1\leqslant i\leqslant D$, and
\begin{gather*}
	E_1\circ E_i=\frac{1}{|X|}(b_{i-1}^*E_{i-1}+a_i^*E_i+c_{i+1}^*E_{i+1}), \qquad 0\leqslant i\leqslant D,
\end{gather*}
where we set $b_{-1}^*E_{-1}=c_{D+1}^*E_{D+1}:=0$. If this is the case, then for $0\leqslant i\leqslant D$, there is a~polynomial $v^*_i \in \mathbb{C}[\xi]$ with $\deg (v^*_i)=i$ such that $v^*_i(|X|E_1)=|X|E_i$, where the multiplication is under~$\circ$. In particular, if we write
\begin{gather}\label{E_1}
	E_1=\frac{1}{|X|}\sum_{i=0}^D\theta_i^* A_i,
\end{gather}
then the $\theta_i^*$ are (real and) mutually distinct. Note also that $\theta_0^*=\operatorname{trace}(E_1)=\operatorname{rank}(E_1)$.

Assume that $\Gamma$ is $Q$-polynomial with respect to the ordering $\{E_i\}_{i=0}^D$. Fix a `base vertex' $x \in X$. For $0\leqslant i \leqslant D$, let $E^*_i=E^*_i(x):=\operatorname{diag}(A_i\hat{x}) \in \operatorname{Mat}_X(\mathbb{C})$. Note that $E^*_iE^*_j=\delta_{i,j}E^*_i$, $\sum\limits^D_{i=0}E^*_i=I$. The $E_i^*$ form a basis for the \emph{dual Bose--Mesner algebra} $M^*=M^*(x)$ of $\Gamma$ with respect to $x$.
We call $A^*=A^*(x):=|X|\operatorname{diag}(E_1\hat{x})\in \operatorname{Mat}_X(\mathbb{C})$ the \emph{dual adjacency matrix} of $\Gamma$ with respect to $x$. We have
\begin{gather*}
	A^*=\sum^D_{i=0}\theta^*_iE^*_i,
\end{gather*}
so that $A^*$ generates $M^*$, and that the $\theta_i^*$ are the eigenvalues of $A^*$, which we call the \emph{dual eigenvalues} of $\Gamma$.
The \emph{Terwilliger} (or \emph{subconstituent}) \emph{algebra} $T=T(x)$ with respect to $x$ is the subalgebra of $\operatorname{Mat}_X(\mathbb{C})$ generated by $M$, $M^*$ \cite{PT.1992(1),PT.1993(1),PT.1993(2)}.
We note that $T$ is again semisimple, is generated by $A$, $A^*$, and that any two non-isomorphic irreducible $T$-modules in $V$ are orthogonal.
The following are relations in $T$ (cf.~\cite[Lemma~3.2]{PT.1992(1)}):
\begin{gather}\label{3products;x}
	E_i^*AE_j^*=E_iA^*E_j= 0 \quad \text{if} \ \ |i-j|>1, \qquad 0\leqslant i,j\leqslant D.
\end{gather}
Observe also that $E_i^*\hat{X}=A_i\hat{x}$, $0\leqslant i\leqslant D$, from which it follows that the $(D+1)$-dimensional subspace
\begin{gather}\label{primary T(x)-module}
	M\hat{x}=M^*\hat{X}=\bigoplus_{i=0}^D \mathbb{C} A_i\hat{x}=\bigoplus_{i=0}^D \mathbb{C} E_i\hat{x}
\end{gather}
of $V$ is an irreducible $T$-module, called the \emph{primary} $T$-module. We note that
\begin{gather}\label{Phi-std}
	A.A_i\hat{x}=b_{i-1}A_{i-1}\hat{x}+a_iA_i\hat{x}+c_{i+1}A_{i+1}\hat{x}, \qquad A^*\!.A_i\hat{x}=\theta_i^*A_i\hat{x}, \qquad 0\leqslant i\leqslant D.
\end{gather}

We refer the reader to \cite{BI,BCN,DKT,CG} for more detailed information.

\section{Dual polar graphs}\label{S:DPG}

Let $D$ be a positive integer. Let $\mathbb{V}$ be one of the following spaces over the f\/inite f\/ield $\mathbb{F}_q$ equipped with a non-degenerate form:\footnote{The scalar $e$ is from \cite[Section~9.4]{BCN} and is assigned to $\mathbb{V}$ to give unif\/ied descriptions to various formulas regarding the dual polar spaces.}
\begin{center}
{
\renewcommand{\arraystretch}{1.2}
\begin{tabular}{cccc}
\hline
space & dimension & form & $e$ \\
\hline \hline
$[C_D(q)]$ & $2D$ & alternating & 1 \\
$[B_D(q)]$ & $2D+1$ & quadratic & 1 \\
$[D_D(q)]$ & $2D$ & quadratic (maximal Witt index $D$) & 0 \\
$[{}^{2\!}D_{D+1}(q)]$ & $2D+2$ & quadratic (non-maximal Witt index $D$) & 2 \\
$[{}^{2\!}A_{2D}(r)]$ & $2D+1$ & Hermitian ($q=r^2$) & $\frac{3}{2}$ \\
$[{}^{2\!}A_{2D-1}(r)]$ & $2D$ & Hermitian ($q=r^2$) & $\frac{1}{2}$ \\
\hline
\end{tabular}
}
\end{center}
A subspace of $\mathbb{V}$ is called (\emph{totally}) \emph{isotropic} if the form vanishes completely on it. We note that maximal isotropic subspaces have dimension $D$. Let $X$ be the set of all maximal isotropic subspaces of $\mathbb{V}$. The \emph{dual polar graph} (on $\mathbb{V}$) has vertex set $X$, where two vertices $x$, $y$ are adjacent if and only if $\dim(x\cap y)=D-1$; cf.~\cite[Section~9.4]{BCN}. We have $\partial(x,y)=D-\dim(x\cap y)$ in general, so that the diameter equals~$D$.

\begin{Assumption}
For the rest of this paper, we will always assume that $\Gamma$ is a dual polar graph with diameter $D \geqslant 3$.
\end{Assumption}

First we summarize some results that we need. The graph $\Gamma$ is distance-regular with intersection numbers
\begin{gather}\label{int.num;x}
	a_i = \big(q^e-1\big) \gauss{i}{1}, \qquad b_i = q^{i+e} \gauss{D-i}{1}, \qquad c_i = \gauss{i}{1}, \qquad 0\leqslant i\leqslant D.
\end{gather}
Note that $a_i = c_ia_1$, $0 \leqslant i \leqslant D$. The eigenvalues of $\Gamma$ are given by
\begin{gather}\label{e-val}
	\theta_i = q^e \gauss{D-i}{1} - \gauss{i}{1}, \qquad 0\leqslant i\leqslant D,
\end{gather}
and $\Gamma$ is $Q$-polynomial with respect to the ordering\footnote{The dual polar graph on $[{}^{2\!}A_{2D-1}(r)]$ has another $Q$-polynomial ordering, which is $\theta_0,\theta_D,\theta_1,\theta_{D-1},\dots$ in terms of the natural ordering $\theta_0>\theta_1>\dots>\theta_D$; cf.~\cite[p.~304]{BI}. However, it turns out that the maximal cliques do not behave well with this ordering (cf.~\cite[Theorem~8.8]{Tanaka2011EJC}), so that we will not pay attention to it in this paper.} $\{\theta_i\}_{i=0}^D$, where $\theta_0>\theta_1>\dots>\theta_D$. Moreover, the corresponding dual eigenvalues are given by
\begin{gather}\label{dual-e-val}
	\theta^*_i = \frac{q(1+q^{D+e-2})}{1-q}
		+ \frac{q(1+q^{D+e-2})(1+q^{D+e-1})}{(q-1)(1+q^{e-1})} q^{-i} ,\qquad 0\leqslant i\leqslant D.
\end{gather}
See \cite[Theorem~9.4.3]{BCN} and \cite[Lemma 16.5]{CW}. The dual polar graph $\Gamma$ is an example of a \emph{regular near polygon} (cf.~\cite[Section~6.4]{BCN}), which means that $\Gamma$ does not have
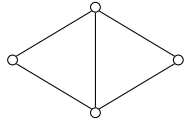
\begin{center}
\begin{tikzpicture}
[bit/.style={circle,inner sep=0mm,minimum size=1.3mm,draw=black}]
\node [bit] (1) {};
\node [bit] (2) [above right=of 1,yshift=-4mm] {};
\node [bit] (3) [below right=of 1,yshift=4mm] {};
\node [bit] (4) [below right=of 2,yshift=4mm] {};
\draw (1.north east) -- (2.south west);
\draw (1.south east) -- (3.north west);
\draw (2.south) -- (3.north);
\draw (2.south east) -- (4.north west);
\draw (3.north east) -- (4.south west);
\end{tikzpicture}
\end{center}
(i.e., $K_{1,1,2}$) as an induced subgraph, and that for every $y\in X$ and a maximal clique $C$, there is a unique $z\in C$ nearest to $y$, provided that $\partial(y,C)<D$.
Note that the former condition implies that every edge lies in a unique maximal clique.
We also note that $\Gamma$ is more specif\/ically a \emph{regular near $2D$-gon},\footnote{The identity $a_i = c_ia_1$, $0 \leqslant i \leqslant D$, mentioned above is a consequence of this; cf.~\cite[Theorem~6.4.1]{BCN}.} i.e., there is in fact no vertex $y$ at distance $D$ from $C$.

Let $C$ be a maximal clique in $\Gamma$.
By the above comments, we have
\begin{gather}\label{clique number}
	|C|=a_1+2=1+q^e,
\end{gather}
which attains the Hof\/fman bound $1-\theta_0\theta_D^{-1}$ (cf.~\cite[Proposition 4.4.6]{BCN}). In other words, $C$ is a~so-called \emph{Delsarte clique}. For $0 \leqslant i \leqslant D-1$, def\/ine the $i^{\mathrm{th}}$ \emph{distance neighbor} of $C$ by
\begin{gather*}
	C_i = \{ y\in X\colon \partial(y,C)=i\}.
\end{gather*}
Then we have
\begin{gather}\label{A_iC}
	A_i\hat{C}=\hat{C}_i+q^e\hat{C}_{i-1}, \qquad 0\leqslant i\leqslant D,
\end{gather}
where $C_{-1}=C_D:=\varnothing$, from which it follows that
\begin{gather}\label{completely regular}
	M\hat{C}=\bigoplus_{i=0}^{D-1}\mathbb{C}\hat{C}_i.
\end{gather}
In other words, $\{C_i\}^{D-1}_{i=0}$ is an \emph{equitable} partition of $X$.
In particular, there are non-negative integers $\widetilde{a}_i$, $\widetilde{b}_i$, $\widetilde{c}_i$, $0 \leqslant i \leqslant D-1$, called the \emph{intersection numbers} of $C$, such that $\widetilde{b}_{D-1}=\widetilde{c}_0=0$, $\widetilde{b}_{i-1}\widetilde{c}_i\ne 0$, $1\leqslant i\leqslant D-1$, and
\begin{gather*}
	\widetilde{a}_i = |C_i\cap\Gamma(y)|, \qquad \widetilde{b}_i = |C_{i+1}\cap\Gamma(y)|, \qquad \widetilde{c}_i = |C_{i-1}\cap\Gamma(y)|
\end{gather*}
for every $y \in C_i$. We may remark that a clique in a distance-regular graph satisf\/ies \eqref{completely regular} precisely when it is a Delsarte clique; cf.~\cite[Section~13.7]{CG}.

\begin{Lemma}\label{rels;int.num}
The following $(i)$, $(ii)$ hold:
\begin{enumerate}\itemsep=0pt
\item[$(i)$] $\widetilde{c}_i = c_i$, $1 \leqslant i \leqslant D-1$.
\item[$(ii)$] $\widetilde{b}_i = b_{i+1}$, $0 \leqslant i \leqslant D-2$.
\end{enumerate}
\end{Lemma}
\begin{proof}
(i) Let $y \in C_i$. Recall that there is a unique $z$ in $C\cap \Gamma_i(y)$ since $\Gamma$ is a regular near polygon. Then we have $C_{i-1}\cap \Gamma(y)=\Gamma_{i-1}(z)\cap\Gamma(y)$, and the result follows.

(ii) With the above notation, pick any $z'\in C$ with $z'\ne z$. Then $z'\in\Gamma_{i+1}(y)$, and we have $C_{i+1}\cap \Gamma(y)=\Gamma_{i+2}(z')\cap\Gamma(y)$, as desired.
\end{proof}

Using \eqref{int.num;x}, Lemma \ref{rels;int.num}, and $\widetilde{a}_i + \widetilde{b}_i + \widetilde{c}_i = \theta_0$, we have
\begin{gather}\label{int num C;q}
	\widetilde{a}_i = q^e \gauss{i+1}{1} - \gauss{i}{1}, \qquad \widetilde{b}_i = q^{i+1+e} \gauss{D-i-1}{1}, \qquad \widetilde{c}_i = \gauss{i}{1}, \qquad 0 \leqslant i \leqslant D-1.
\end{gather}

We now recall the \emph{Terwilliger algebra} of $\Gamma$ \emph{with respect to} $C$ in the sense of Suzuki~\cite{HS}.
For $0 \leqslant i \leqslant D-1$, let $\widetilde{E}^*_i=\widetilde{E}^*_i(C):=\operatorname{diag}(\hat{C}_i) \in \operatorname{Mat}_X(\mathbb{C})$. Note that $\widetilde{E}^*_i \widetilde{E}^*_j = \delta_{i,j}\widetilde{E}^*_i$, $\sum\limits^{D-1}_{i=0}\widetilde{E}^*_i = I$. The $\widetilde{E}_i^*$ form a basis for the \emph{dual Bose--Mesner algebra} $\widetilde{M}^*=\widetilde{M}^*(C)$ of $\Gamma$ with respect to $C$. The \emph{dual adjacency matrix} of $\Gamma$ with respect to $C$ is def\/ined by (cf.~\cite[equation~(50)]{JHL})
\begin{gather*}
	\widetilde{A}^*=\widetilde{A}^*(C):=\frac{|X|}{|C|}\operatorname{diag}(E_1\hat{C})=\frac{1}{|C|}\sum_{y\in C}A^*(y).
\end{gather*}
From \eqref{E_1}, \eqref{clique number}, and \eqref{A_iC}, it follows that
\begin{gather}\label{E_1C}
	\frac{1}{|C|}E_1\hat{C} = \frac{1}{|X|}\sum^{D-1}_{i=0}\widetilde{\theta}^*_i\hat{C}_i,
\end{gather}
where
\begin{gather*}
	\widetilde{\theta}^*_i=\frac{1}{1+q^e}\theta^*_i+\frac{q^e}{1+q^e}\theta^*_{i+1}, \qquad 0\leqslant i\leqslant D-1,
\end{gather*}
so that (cf.~\cite[Lemma 4.11]{JHL})
\begin{gather*}
	\widetilde{A}^* = \sum^{D-1}_{i=0}\widetilde{\theta}^*_i\widetilde{E}^*_i.
\end{gather*}
By \eqref{dual-e-val}, we have
\begin{gather}\label{dual-e-val;tilde}
	\widetilde{\theta}^*_i = \frac{q(1+q^{D+e-2})}{1-q} +
			\frac{q(1+q^{D+e-2})(1+q^{D+e-1})}{(q-1)(1+q^e)} q^{-i}, \qquad 0\leqslant i\leqslant D-1.
\end{gather}
In particular, $\widetilde{A}^*$ has $D$ mutually distinct real eigenvalues and hence generates $\widetilde{M}^*$. The \emph{Terwilliger algebra} $\widetilde{T}=\widetilde{T}(C)$ with respect to $C$ is the subalgebra of $\operatorname{Mat}_X(\mathbb{C})$ generated by~$M$,~$\widetilde{M}^*$.
We note that $\widetilde{T}$ is semisimple and is generated by $A$, $\widetilde{A}^*$. By virtue of \eqref{3products;x}, the following are relations in $\widetilde{T}$:
\begin{gather}\label{3products;C}
	\widetilde{E}_i^*A\widetilde{E}_j^*=E_i\widetilde{A}^*E_j= 0 \quad \text{if} \ \ |i-j|>1, \qquad 0\leqslant i,j\leqslant D,
\end{gather}
where we set $\widetilde{E}_D^*:=0$ for convenience. The subspace \eqref{completely regular} of $V$ is an irreducible $\widetilde{T}$-module with dimension $D$, called the \emph{primary} $\widetilde{T}$-module. We note that
\begin{gather}\label{Phit-std}
	A.\hat{C}_i=\widetilde{b}_{i-1}\hat{C}_{i-1}+\widetilde{a}_i\hat{C}_i+\widetilde{c}_{i+1}\hat{C}_{i+1}, \qquad \widetilde{A}^*\!.\hat{C}_i=\widetilde{\theta}_i^*\hat{C}_i, \qquad 0\leqslant i\leqslant D-1.
\end{gather}

\begin{Remark}The Bose--Mesner algebra $M$ coincides in this case with the commutant of the corresponding classical group acting on~$X$, whereas the Terwilliger algebras $T$, $\widetilde{T}$ are subalgebras of those of maximal parabolic subgroups. See also \cite{Stanton1980AJM,Stanton1981GD}.
\end{Remark}

\section{The algebra $\mathbf{T}$} \label{S:algT}

We continue to discuss the dual polar graph $\Gamma$.

\begin{Assumption}For the rest of this paper, we will f\/ix a vertex $x \in X$ and a maximal clique~$C$ containing $x$.
\end{Assumption}

\begin{Definition}[{\cite[Def\/inition 5.20]{JHL}}]The \emph{generalized Terwilliger algebra} of $\Gamma$ with respect to~$x$, $C$ is the subalgebra $\mathbf{T}=\mathbf{T}(x,C)$ of $\operatorname{Mat}_X(\mathbb{C})$ generated by $T=T(x)$, $\widetilde{T}=\widetilde{T}(C)$.
\end{Definition}

We note that $\mathbf{T}$ is semisimple and is generated by $A$, $A^*$, and $\widetilde{A}^*$, where $A^*\widetilde{A}^*=\widetilde{A}^*A^*$.
We now ref\/ine the equitable partition $\{C_i\}^{D-1}_{i=0}$ of $X$ as follows
\begin{gather*}
	C^-_i =\Gamma_i(x) \cap C_i, \qquad C^+_i = \Gamma_{i+1}(x) \cap C_i, \qquad 0 \leqslant i \leqslant D-1.
\end{gather*}
See Fig.~\ref{new partition when D=4}. For notational convenience, we set
\begin{gather}\label{empty in some cases}
	C^-_{-1}=C^+_{-1}=C^-_D=C^+_D:=\varnothing.
\end{gather}
Observe that $C_i = C^-_i \cup C^+_i$, $0 \leqslant i \leqslant D-1$, and that $\Gamma_i(x)=C^{+}_{i-1} \cup C^-_i$, $0 \leqslant i \leqslant D$.

\begin{figure}
\centering
\scalemath{0.6}{
\begin{tikzpicture}
 [scale=1,thick,auto=left,every node/.style={circle,draw}]
 \node (n1) at (0,0) {$C^-_0$};
 \node (n2) at (0,2) {$C^+_0$};
 \node (n3) at (2,2) {$C^-_1$};
 \node (n4) at (2,4) {$C^+_1$};
 \node (n5) at (4,4) {$C^-_2$};
 \node (n6) at (4,6) {$C^+_2$};
 \node (n7) at (6,6) {$C^-_3$};
 \node (n8) at (6,8) {$C^+_3$};

 \node[fill=black!10] (c0) at (0,-2.5) {${~}C^{~}_0$};
 \node[fill=black!10] (c1) at (2,-2.5) {${~}C^{~}_1$};
 \node[fill=black!10] (c2) at (4,-2.5) {${~}C^{~}_2$};
 \node[fill=black!10] (c3) at (6,-2.5) {${~}C^{~}_3$};

 \node[fill=blue!20] (r0) at (-2.5,0) {${~}\Gamma_0$};
 \node[fill=blue!20] (r1) at (-2.5,2) {${~}\Gamma_1$};
 \node[fill=blue!20] (r2) at (-2.5,4) {${~}\Gamma_2$};
 \node[fill=blue!20] (r3) at (-2.5,6) {${~}\Gamma_3$};
 \node[fill=blue!20] (r4) at (-2.5,8) {${~}\Gamma_4$};

 \foreach \from/\to in {n1/n2,n3/n4,n5/n6,n7/n8,n1/n3,n3/n5,n5/n7,n2/n4,n4/n6,n6/n8, c0/c1,c1/c2,c2/c3, r0/r1, r1/r2, r2/r3, r3/r4}
 \draw (\from) -- (\to)[line width=0.7mm] ;

 \draw (c0) -- (n1) [dashed];
 \draw (c1) -- (n3) [dashed];
 \draw (c2) -- (n5) [dashed];
 \draw (c3) -- (n7) [dashed];

 \draw (n2) -- (n3) [dashed];
 \draw (n4) -- (n5) [dashed];
 \draw (n6) -- (n7) [dashed];

 \draw (r0) -- (n1) [dashed];
 \draw (r1) -- (n2) [dashed];
 \draw (r2) -- (n4) [dashed];
 \draw (r3) -- (n6) [dashed];
 \draw (r4) -- (n8) [dashed];

\end{tikzpicture}}
\caption{The partition $\{C^{\pm}_i\}^{D-1}_{i=0}$ when $D=4$.}\label{new partition when D=4}
\end{figure}
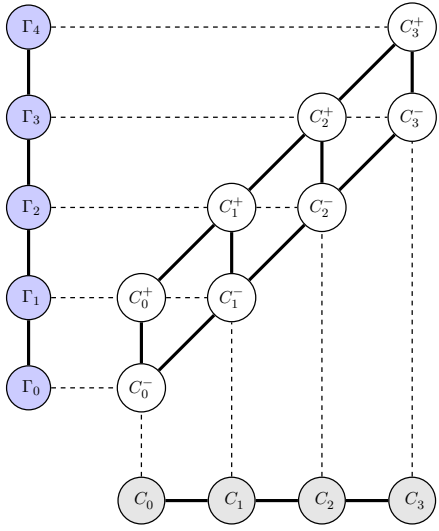

\begin{Lemma} \label{int.num;W}
The following $(i)$, $(ii)$ hold:
\begin{enumerate}\itemsep=0pt
\item[$(i)$] For $0 \leqslant i \leqslant D-1$ and $z \in C^-_i$, we have
\begin{center}
{
\renewcommand{\arraystretch}{1.2}
\begin{tabular}{c||@{\quad}c@{\qquad}c@{\qquad}c@{\qquad}c@{\qquad}c@{\quad}}
$Y$ & $C^-_{i-1}$ & $C^+_{i-1}$ & $C^-_i$ & $C^+_i$ & $C^-_{i+1}$\\
\hline
$|\Gamma(z) \cap Y|$ & $c_i$ & $0$ & $a_i$ & $b_i - b_{i+1}$ & $b_{i+1}$
\end{tabular}
}.
\end{center}

\item[$(ii)$] For $0 \leqslant i \leqslant D-1$ and $z \in C^+_i$, we have
\begin{center}
{
\renewcommand{\arraystretch}{1.2}
\begin{tabular}{c||@{\quad}c@{\qquad}c@{\qquad}c@{\qquad}c@{\qquad}c@{\quad}}
$Y$ & $C^+_{i-1}$ & $C^-_i$ & $C^+_i$ & $C^-_{i+1}$ & $C^+_{i+1}$\\
\hline
$|\Gamma(z) \cap Y|$ & $c_i$ & $c_{i+1}-c_i$ & $a_{i+1}$ & $0$ & $b_{i+1}$
\end{tabular}}.
\end{center}
\end{enumerate}
In particular, the partition $\{C^{\pm}_i\}^{D-1}_{i=0}$ is again equitable.\footnote{See \cite[Proposition 1.3]{HS2007EJC} for a more general result.}
\end{Lemma}
\begin{proof}
Use Lemma \ref{rels;int.num}.
\end{proof}

\begin{Lemma}\label{|C|;q-exp}
We have
\begin{gather*}
	|C^-_i|=q^{ie}\prod^i_{n=1}\frac{q^D-q^n}{q^n-1}, \qquad |C^+_i|=q^{(i+1)e}\prod^i_{n=1}\frac{q^D-q^n}{q^n-1},
\qquad 0\leqslant i \leqslant D-1.
\end{gather*}
In particular, the $C^{\pm}_i$ are non-empty.
\end{Lemma}

\begin{proof}
It follows from Lemma \ref{int.num;W} that
\begin{gather*}
	c_i|C_i^-|=b_i|C_{i-1}^-|, \qquad c_i|C_i^+|=b_i|C_{i-1}^+|.
\end{gather*}
The result follows from this and \eqref{int.num;x}, together with $|C_0^-|=1$ and $|C_0^+|=q^e$ (cf.~\eqref{clique number}).
\end{proof}

Let $\mathbf{W}$ be the linear span of the $\hat{C}^{\pm}_i$.
Thus, $\mathbf{W}$ has the following ordered orthogonal basis:
\begin{gather}\label{1st basis}
	\mathcal{C} = \big\{\hat{C}^-_0, \hat{C}^+_0, \hat{C}^-_1, \hat{C}^+_1, \dots, \hat{C}^-_{D-1}, \hat{C}^+_{D-1}\big\}.
\end{gather}

\begin{Proposition}\label{W:irrT-mod}
The subspace $\mathbf{W}$ is an irreducible $\mathbf{T}$-module.
\end{Proposition}
\begin{proof}
It is clear that $\mathbf{W}$ is closed under the actions of the $E_i^*$ and the $\widetilde{E}_i^*$.
Moreover, since $\{C^{\pm}_i\}^{D-1}_{i=0}$ is an equitable partition by Lemma \ref{int.num;W}, $\mathbf{W}$ is also $A$-invariant.
It follows that $\mathbf{W}$ is a $\mathbf{T}$-module.
We now show the irreducibility.
Since $\mathbf{T}$ is semisimple,
$\mathbf{W}$ is an orthogonal direct sum of irreducible $\mathbf{T}$-modules.
Among these $\mathbf{T}$-modules, there is one which is not orthogonal to $\mathbb{C}\hat{X}=E_0\mathbf{W}$, denoted by $\mathbf{W}_0$.
Then we have $0 \ne E_0\mathbf{W}_0 \subset E_0\mathbf{W}$, so that $\hat{X}\in E_0\mathbf{W}_0 \subset \mathbf{W}_0$, from which it follows that $\mathbf{W}=M^*\widetilde{M}^*\hat{X} \subset \mathbf{W}_0$.
Hence $\mathbf{W}_0=\mathbf{W}$, and the result follows.
\end{proof}

We call $\mathbf{W}$ the \emph{primary} $\mathbf{T}$-module.
Note that $\mathbf{W}$ contains both the primary $T$-module $M\hat{x}$ and the primary $\widetilde{T}$-module $M\hat{C}$.
Let $M\hat{x}^{\perp}$ (resp.~$M\hat{C}^{\perp}$) be the orthogonal complement of $M\hat{x}$ (resp.~$M\hat{C}$) in $\mathbf{W}$.
In Sections \ref{S:T-mod W} and \ref{S:T(C)-mod}, we will show that $M\hat{x}^{\perp}$ (resp.~$M\hat{C}^{\perp}$) is also an irreducible $T$-module (resp.~$\widetilde{T}$-module).
Thus, $\mathbf{W}$ decomposes in two ways:
\begin{align}
\mathbf{W} 	&= M\hat{x} \oplus M\hat{x}^{\perp} \quad ~
	\textrm{(orthogonal direct sum of irreducible $T$-modules)}\label{ods;T}\\
	&= M\hat{C} \oplus M\hat{C}^{\perp} \quad
	\textrm{(orthogonal direct sum of irreducible $\widetilde{T}$-modules)}.\label{ods;Ttil}
\end{align}

We end this section by describing the actions of $A$, $A^*$, and $\widetilde{A}^*$ on $\mathbf{W}$ in terms of the basis $\mathcal{C}$.

\begin{Lemma}\label{actA}
For $0 \leqslant i \leqslant D-1$, we have
\begin{gather*}
A.\hat{C}^-_i =
\frac{q^{D+e}-q^{i+e}}{q-1} \hat{C}^-_{i-1}
+ \big(q^e-1\big)\frac{q^i-1}{q-1} \hat{C}^-_i
+ q^i \hat{C}^+_i
+ \frac{q^{i+1}-1}{q-1} \hat{C}^-_{i+1},\\
A.\hat{C}^+_i =
\frac{q^{D+e}-q^{i+e}}{q-1}\hat{C}^+_{i-1}
+ q^{e+i}\hat{C}^-_i
+ \big(q^e-1\big)\frac{q^{i+1}-1}{q-1}\hat{C}^+_i
+ \frac{q^{i+1}-1}{q-1}\hat{C}^+_{i+1}.
\end{gather*}
\end{Lemma}

\begin{proof}
From Lemma \ref{int.num;W} it follows that
\begin{gather*}
A.\hat{C}^-_i = b_i \hat{C}^-_{i-1}
+ a_i \hat{C}^-_i
+ (c_{i+1}-c_i) \hat{C}^+_i
+ c_{i+1} \hat{C}^-_{i+1},\\
A.\hat{C}^+_i = b_i\hat{C}^+_{i-1}
+ (b_i-b_{i+1})\hat{C}^-_i
+ a_{i+1}\hat{C}^+_i
+ c_{i+1}\hat{C}^+_{i+1}
\end{gather*}
for $0 \leqslant i \leqslant D-1$.
Evaluate the two identities using \eqref{int.num;x}.
\end{proof}

\begin{Lemma}\label{act.A*,A*t}
For $0 \leqslant i \leqslant D-1$, we have
\begin{gather*}
	A^*.\hat{C}^-_i = \theta_i^*\hat{C}^-_i, \qquad A^*.\hat{C}^+_i= \theta_{i+1}^*\hat{C}^+_i, \qquad \widetilde{A}^*.\hat{C}^{\pm}_i = \widetilde{\theta}_i^*\hat{C}^{\pm}_i.
\end{gather*}
\end{Lemma}

\begin{proof}
Clear.
\end{proof}

\section[Leonard systems of dual $q$-Krawtchouk type]{Leonard systems of dual $\boldsymbol{q}$-Krawtchouk type}\label{S:pre LS}

Let $d$ be a positive integer, and let $\mathsf{W}$ be a vector space over $\mathbb{C}$ with $\dim (\mathsf{W})=d+1$. An element $\mathsf{A}\in \operatorname{End}(\mathsf{W})$ is called \emph{multiplicity-free} if it has $d+1$ mutually distinct eigenvalues. Suppose that~$\mathsf{A}$ is multiplicity-free, and let $\{\mytheta_i\}_{i=0}^d$ be an ordering of the eigenvalues of~$\mathsf{A}$. Then there is a sequence of elements $\{\mathsf{E}_i\}_{i=0}^d$ in $\operatorname{End}(\mathsf{W})$ such that $\mathsf{A}\mathsf{E}_i=\mytheta_i\mathsf{E}_i$, $\mathsf{E}_i\mathsf{E}_j = \delta_{i,j}\mathsf{E}_i$, $\sum\limits^d_{i=0}\mathsf{E}_i=\mathsf{I}$, where~$\mathsf{I}$ is the identity of $\operatorname{End}(\mathsf{W})$. We call~$\mathsf{E}_i$ the \emph{primitive idempotent} of~$\mathsf{A}$ associated with $\mytheta_i$.

\begin{Definition}
A \emph{Leonard system} on $\mathsf{W}$ is a sequence
\begin{gather*}
	\mathsf{\Phi}=\big(\mathsf{A};\mathsf{A}^*;\{\mathsf{E}_i\}^d_{i=0};\{\mathsf{E}^*_i\}^d_{i=0}\big)
\end{gather*}
that satisf\/ies the following axioms (LS\ref{LS1})--(LS\ref{LS4}):
\begin{enumerate}[(LS1)]\itemsep=0pt
\item Each of $\mathsf{A}, \mathsf{A}^*$ is a multiplicity-free element in $\mathrm{End}(\mathsf{W})$. \label{LS1}
\item $\{\mathsf{E}_i\}^d_{i=0}$ (resp.~$\{\mathsf{E}^*_i\}^d_{i=0}$) is an ordering of the primitive idempotents of $\mathsf{A}$ (resp.~$\mathsf{A}^*$).
\item $\mathsf{E}^*_i\mathsf{A}\mathsf{E}^*_j=\mathsf{E}_i\mathsf{A}^*\mathsf{E}_j=0$ if $|i-j|>1$, $0 \leqslant i,j \leqslant d$. \label{LS3}
\item There is no proper subspace of $\mathsf{W}$ which is both $\mathsf{A}$- and $\mathsf{A}^*$-invariant. \label{LS4}
\end{enumerate}
We call $d$ the \emph{diameter} of $\mathsf{\Phi}$.
\end{Definition}

We note that the above def\/inition is taken from \cite[Def\/inition 2.1]{ITT2001P} (and the paragraph following it), and is easily shown to be equivalent to the original def\/inition in \cite[Def\/ini\-tion~1.4]{PT.2001}.

Let $\mathsf{\Phi} = (\mathsf{A};\mathsf{A}^*;\{\mathsf{E}_i\}^d_{i=0};\{\mathsf{E}^*_i\}^d_{i=0})$ be a Leonard system on $\mathsf{W}$.
Note that each of the following is also a Leonard system on $\mathsf{W}$:
\begin{gather}\label{relatives}
	\mathsf{\Phi}^* := \big(\mathsf{A}^*;\mathsf{A};\{\mathsf{E}^*_i\}^d_{i=0};\{\mathsf{E}_i\}^d_{i=0}\big), \qquad \mathsf{\Phi}^{\Downarrow} := \big(\mathsf{A};\mathsf{A}^*;\{\mathsf{E}_{d-i}\}^d_{i=0};\{\mathsf{E}^*_i\}^d_{i=0}\big).
\end{gather}
A Leonard system $\mathsf{\Psi}$ on a vector space $\mathsf{W}'$ is \emph{isomorphic} to $\mathsf{\Phi}$ if there is a $\mathbb{C}$-algebra isomorphism $\sigma\colon \mathrm{End}(\mathsf{W}) \to \mathrm{End}(\mathsf{W}')$ such that
$\mathsf{\Psi}=\mathsf{\Phi}^{\sigma} :=
(\mathsf{A}^{\sigma};\mathsf{A}^{*\sigma};\{\mathsf{E}^{\sigma}_i\}^d_{i=0};\{\mathsf{E}^{*\sigma}_i\}^d_{i=0})$.

For $0 \leqslant i \leqslant d$, let $\mytheta_i$ (resp.~$\mytheta^*_i$) be the eigenvalue of $\mathsf{A}$ (resp.~$\mathsf{A}^*$) associated with $\mathsf{E}_i$ (resp.~$\mathsf{E}^*_i$).
Then there are non-zero scalars $\{\myvarphi_i\}^d_{i=1}$ in $\mathbb{C}$ and a $\mathbb{C}$-algebra isomorphism $\natural\colon \mathrm{End}(\mathsf{W})\to \operatorname{Mat}_{d+1}(\mathbb{C})$ (the full matrix algebra) such that \cite[Theorem 3.2]{PT.2001}
\begin{gather*}
\mathsf{A}^{\natural} =
	\begin{pmatrix}
	\mytheta_0 & & & & & \mathbf{0}\\
	1 & \mytheta_1 & \\
	& 1 & \mytheta_2 & \\
	& & \cdot & \cdot & \\
	& & & \cdot & \cdot \\
	\mathbf{0} & & & & 1 & \mytheta_d
	\end{pmatrix},
\qquad
\mathsf{A}^{*\natural} =
	\begin{pmatrix}
	\mytheta^*_0 & \myvarphi_1 & & & & \mathbf{0}\\
	& \mytheta^*_1 & \myvarphi_2\\
	& & \mytheta^*_2 & \cdot\\
	& & & \cdot & \cdot\\
	& & & & \cdot & \myvarphi_d \\
	\mathbf{0} & & & & & \mytheta^*_d
	\end{pmatrix}.
\end{gather*}
We call $\{\myvarphi_i\}_{i=1}^d$ the \emph{first split sequence} of $\mathsf{\Phi}$. Let $\{\myphi_i\}_{i=1}^d$ be the f\/irst split sequence of $\mathsf{\Phi}^{\Downarrow}$ and call this the \emph{second split sequence} of $\mathsf{\Phi}$. The \emph{parameter array} of $\mathsf{\Phi}$ is the sequence
\begin{gather*}
	\big(\{\mytheta_i\}^d_{i=0};\{\mytheta^*_i\}^d_{i=0};\{\myvarphi_i\}^d_{i=1};\{\myphi_i\}^d_{i=1}\big).
\end{gather*}
It is clear that the parameter array is a complete invariant for the isomorphism classes of Leonard systems.
Terwilliger \cite[Section~5]{PT.2005} displayed all the parameter arrays of Leonard systems in parametric form.
We now recall the dual $q$-Krawtchouk family of Leonard systems on which we will focus.

\begin{Definition}[{\cite[Example~5.9]{PT.2005}}]\label{dual q PA}
The Leonard system $\mathsf{\Phi}$ is said to have \emph{dual $q$-Krawtchouk type} if there are scalars $\myalpha$, $\myalpha^*$, $\mybeta$, $\mybeta^*$, $\mygamma$ with $\mybeta\mybeta^*\mygamma\ne 0$ such that
\begin{gather*}
	\mytheta_i = \myalpha + \mybeta q^{-i} + \mygamma q^i, \qquad
	\mytheta^*_i = \myalpha^* + \mybeta^* q^{-i}
\end{gather*}
for $0 \leqslant i \leqslant d$, and
\begin{gather*}
	\myvarphi_i = \mybeta\mybeta^* q^{1-2i}\big(1-q^i\big)\big(1-q^{i-d-1}\big),\qquad
	\myphi_i = \mygamma\mybeta^* q^{d+1-2i}\big(1-q^i\big)\big(1-q^{i-d-1}\big)
\end{gather*}
for $1 \leqslant i \leqslant d$. We call $(\myalpha,\myalpha^*,\mybeta,\mybeta^*,\mygamma\,;q,d)$ the \emph{parameter sequence} of $\mathsf{\Phi}$.
\end{Definition}

Take a non-zero vector $\mathsf{u} \in \mathsf{E}_0\mathsf{W}$. By \cite[Lemma 10.2]{Terwilliger2004LAA}, the vectors $\{\mathsf{E}^*_i\mathsf{u}\}^d_{i=0}$ form a basis for $\mathsf{W}$, called a $\mathsf{\Phi}$-\emph{standard} \emph{basis}. In view of (LS\ref{LS3}), (LS\ref{LS4}), there are scalars $\mathsf{a}_i$, $\mathsf{b}_i$, $\mathsf{c}_i$, $0\leqslant i\leqslant d$ such that $\mathsf{b}_d=\mathsf{c}_0=0$, $\mathsf{b}_{i-1}\mathsf{c}_i\ne 0$, $1\leqslant i\leqslant d$, and
\begin{gather}\label{recurrence 1}
	\mathsf{A}.\mathsf{E}^*_i\mathsf{u} =
	\mathsf{b}_{i-1}\mathsf{E}^*_{i-1}\mathsf{u} + \mathsf{a}_i\mathsf{E}^*_i\mathsf{u} + \mathsf{c}_{i+1}\mathsf{E}^*_{i+1}\mathsf{u},
	\qquad 0\leqslant i \leqslant d,
\end{gather}
where $\mathsf{b}_{-1}\mathsf{E}^*_{-1}\mathsf{u}=\mathsf{c}_{d+1}\mathsf{E}^*_{d+1}\mathsf{u}:=0$.
We call $\mathsf{a}_i$, $\mathsf{b}_i$, $\mathsf{c}_i$ the \emph{intersection numbers} of $\mathsf{\Phi}$.
Observe that
\begin{gather}\label{ai+bi+ci}
	\mathsf{a}_i + \mathsf{b}_i + \mathsf{c}_i = \mytheta_0, \qquad 0 \leqslant i \leqslant d.
\end{gather}
By \cite[Theorem 17.7]{Terwilliger2004LAA}, we have
\begin{gather}
	\mathsf{b}_i = \myvarphi_{i+1}
		\frac{(\mytheta^*_i-\mytheta^*_0)(\mytheta^*_i-\mytheta^*_1)\cdots (\mytheta^*_i-\mytheta^*_{i-1})}
		{(\mytheta^*_{i+1}-\mytheta^*_0)(\mytheta^*_{i+1}-\mytheta^*_1)\cdots (\mytheta^*_{i+1}-\mytheta^*_i)}
		, \qquad 0 \leqslant i \leqslant d-1, \label{bi formula}\\
	\mathsf{c}_i = \myphi_i
		\frac{(\mytheta^*_i-\mytheta^*_{i+1}) (\mytheta^*_i-\mytheta^*_{i+2}) \cdots (\mytheta^*_i-\mytheta^*_{d})}
		{(\mytheta^*_{i-1}-\mytheta^*_i) (\mytheta^*_{i-1}-\mytheta^*_{i+1}) \cdots (\mytheta^*_{i-1}-\mytheta^*_d)}, \qquad
		 1 \leqslant i \leqslant d. \label{ci formula}
\end{gather}

Def\/ine a f\/inite sequence of polynomials $\{\mathsf{v}_i\}_{i=0}^d$ in $\mathbb{C}[\xi]$ with $\deg(\mathsf{v}_i)=i$ by $\mathsf{v}_0 :=1$ and
\begin{gather}\label{recurrence 2}
	\xi \mathsf{v}_i = \mathsf{b}_{i-1}\mathsf{v}_{i-1} + \mathsf{a}_i\mathsf{v}_i + \mathsf{c}_{i+1}\mathsf{v}_{i+1}, \qquad 0 \leqslant i \leqslant d-1,
\end{gather}
where $\mathsf{b}_{-1}\mathsf{v}_{-1}:=0$.
From \eqref{recurrence 1} and \eqref{recurrence 2} it follows that (cf.~\cite[Theorem 13.4]{Terwilliger2004LAA})
\begin{gather*}
	\mathsf{v}_i(\mathsf{A}).\mathsf{E}^*_0 \mathsf{u} = \mathsf{E}^*_i \mathsf{u}, \qquad 0\leqslant i\leqslant d.
\end{gather*}
Consider the following normalization:
\begin{gather*}
	\mathsf{f}_i := \frac{\mathsf{v}_i}{\mathsf{v}_i(\mytheta_0)}, \qquad 0\leqslant i\leqslant d,
\end{gather*}
where using \eqref{ai+bi+ci}, \eqref{recurrence 2}, we have by induction (cf.~\cite[Lemma 13.2]{Terwilliger2004LAA})
\begin{gather}\label{ki}
	\mathsf{v}_i(\mytheta_0) = \frac{\mathsf{b}_0\mathsf{b}_1\cdots\mathsf{b}_{i-1}}{\mathsf{c}_1\mathsf{c}_2\cdots\mathsf{c}_i}(\ne 0),
	\qquad 0 \leqslant i \leqslant d.
\end{gather}
By \cite[Theorem 17.4]{Terwilliger2004LAA}, for $0\leqslant i\leqslant d$ we have
\begin{gather*}
	\mathsf{f}_i = \sum^i_{n=0} \frac{(\mytheta^*_i-\mytheta^*_0)(\mytheta^*_i-\mytheta^*_1)\cdots(\mytheta^*_i-\mytheta^*_{n-1})(\xi-\mytheta_0)(\xi-\mytheta_1)\cdots (\xi-\mytheta_{n-1})}{\myvarphi_1\myvarphi_2 \cdots \myvarphi_n}.
\end{gather*}

Def\/ine the scalars $\mathsf{m}_i$, $0\leqslant i\leqslant d$, by
\begin{gather}\label{mi}
	\mathsf{m}_i=\operatorname{trace}(\mathsf{E}_i\mathsf{E}_0^*), \qquad 0\leqslant i\leqslant d,
\end{gather}
so that $\mathsf{E}_0^*\mathsf{E}_i\mathsf{E}_0^*=\mathsf{m}_i\mathsf{E}_0^*$, $0\leqslant i\leqslant d$; cf.~\cite[Lemma 9.2]{Terwilliger2004LAA}.
By \cite[Theorem 17.12]{Terwilliger2004LAA}, we have
\begin{gather*}
	\mathsf{m}_i=\frac{\myvarphi_1\myvarphi_2 \cdots \myvarphi_i\myphi_1\myphi_2 \cdots \myphi_{d-i}}{(\mytheta^*_0-\mytheta^*_1)\cdots (\mytheta^*_0-\mytheta^*_d) (\mytheta_i-\mytheta_0)\cdots(\mytheta_i-\mytheta_{i-1})(\mytheta_i-\mytheta_{i+1})\cdots(\mytheta_i-\mytheta_d)}.
\end{gather*}

Assume now that $\mathsf{\Phi}$ has dual $q$-Krawtchouk type as in Def\/inition \ref{dual q PA}.
Then we have\footnote{See also \cite[Theorem 2.1]{PT.1992(1)}. In the notation of \cite[Theorem 2.1]{PT.1992(1)}, the Leonard systems of dual $q$-Krawtchouk type correspond to Case (I) with $s^*=r_1=r_2=0$.}
\begin{gather}\label{bi, ci : dual-qK}
	\mathsf{b}_i = \mybeta\big(1-q^{i-d}\big), \qquad \mathsf{c}_i = \mygamma\big(1-q^i\big), \qquad 0\leqslant i\leqslant d.
\end{gather}
The values of the $\mathsf{f}_i$ at $\xi=\mytheta_j$ are given by (cf.~\cite[Example 5.9]{PT.2005})
\begin{gather}\label{dual q-Krawtchouk}
	\mathsf{f}_i(\mytheta_j) = {}_{3}\phi_{2}\!
	\left(
	\left.
	\begin{matrix}
	q^{-i}, \, q^{-j}, \, \mybeta^{-1}\mygamma q^j \\
	0, \, q^{-d}
	\end{matrix}\,
	\right| q,q
	\right), \qquad 0\leqslant i,j\leqslant d.
\end{gather}
It follows that the $\mathsf{f}_i$ are the \emph{dual $q$-Krawtchouk polynomials} \cite[Section~14.17]{KLS2010B} in the variable
\begin{gather*}
	\xi=\xi(x):=\myalpha + \mybeta q^{-x} + \mygamma q^x.
\end{gather*}
We also have
\begin{gather}\label{mi formula}
	\mathsf{m}_i = \frac{\mybeta^i}{\mygamma^i} \frac{(q^{-d};q)_i (1-\mybeta\mygamma^{-1} q^{-2i})}{q^{i(i-1)} (q;q)_i (\mybeta\mygamma^{-1} q^{-d-i};q)_{d+1}}, \qquad 0\leqslant i\leqslant d.
\end{gather}

Following \cite{KB2011AA}, we f\/ix a non-zero scalar $\mytau$ such that
\begin{gather}\label{tau}
	\mytau^2=\mybeta^{-1}\mygamma,
\end{gather}
and renormalize the $\mathsf{f}_i$ so that they are monic\footnote{A Laurent polynomial in $\eta$ is called \emph{monic} if the coef\/f\/icient of its highest degree term in $\eta$ equals $1$.} as symmetric Laurent polynomials in $\eta$:
\begin{gather}\label{monic symmetric Laurent dual q-Krawtchouk}
	\mathsf{h}_i(\eta):=\frac{(q^{-d};q)_i}{\mytau^i} {}_{3}\phi_{2}\!\left(\left.
	\begin{matrix}
	q^{-i},\, \mytau\eta^{-1},\, \mytau\eta \\
	0, \, q^{-d}
	\end{matrix}\,
	\right| q,q \right),
	\qquad 0\leqslant i\leqslant d.
\end{gather}
The $\mathsf{h}_i$ depend on the parameters $q$, $d$, and $\mytau$, and we will write
\begin{gather}\label{hi in 3 parameters}
	\mathsf{h}_i(\eta)=h_i(\eta;\mytau,d;q), \qquad 0\leqslant i\leqslant d.
\end{gather}
We note that $\mathsf{h}_i$ has highest degree $i$ and lowest degree $-i$ in $\eta$, and that
\begin{gather}\label{hi and fi}
	\mathsf{h}_i(\eta)=\frac{(q^{-d};q)_i}{\mytau^i}\mathsf{f}_i\bigl(\xi\big(\log_q\mytau^{-1}\eta\big)\bigr), \qquad 0\leqslant i\leqslant d.
\end{gather}

Let $\Wsf$ be a vector space over $\mathbb{C}$ containing $\mathsf{W}$ as a subspace, and let $\mathsf{X}\in\operatorname{End}(\Wsf)$ be invertible such that $(\mathsf{X}+\mathsf{X}^{-1})\mathsf{W}\subset\mathsf{W}$ and that
\begin{gather*}
	\mathsf{A}=\myalpha\mathsf{I}+\mybeta\mytau \big(\mathsf{X}+\mathsf{X}^{-1}\big)=\myalpha\mathsf{I}+\mybeta\big(\mytau^{-1}\mathsf{X}\big)^{-1}+\mygamma\big(\mytau^{-1}\mathsf{X}\big)
\end{gather*}
holds on $\mathsf{W}$.
Then it follows from \eqref{hi and fi} that
\begin{gather*}
	\mathsf{h}_i(\mathsf{X})|_{\mathsf{W}}=\frac{(q^{-d};q)_i}{\mytau^i}\mathsf{f}_i(\mathsf{A})=\mytau^i(q;q)_i\mathsf{v}_i(\mathsf{A}),
\end{gather*}
where we have also used \eqref{ki} and \eqref{bi, ci : dual-qK}.
In particular, we have
\begin{gather}\label{hi on std basis}
	\mathsf{h}_i(\mathsf{X}).\mathsf{E}^*_0 \mathsf{u} = \mytau^i(q;q)_i\mathsf{E}^*_i \mathsf{u}, \qquad 0\leqslant i\leqslant d.
\end{gather}

\section[Irreducible $T$-submodules of $\mathbf{W}$]{Irreducible $\boldsymbol{T}$-submodules of $\boldsymbol{\mathbf{W}}$}\label{S:T-mod W}

Recall the primary $\mathbf{T}$-module $\mathbf{W}$. In this section, we focus on the decomposition \eqref{ods;T} as $T$-modules. We f\/irst discuss the primary $T$-module $M\hat{x}$. Consider the sequence
	\begin{gather*}
	\Phi:=\big(A;A^*;\{E_i\}^D_{i=0};\{E^*_i\}^D_{i=0}\big)\big|_{M\hat{x}},
	\end{gather*}
where $|_{M\hat{x}}$ means that each of the elements in the sequence is restricted to $M\hat{x}$.

\begin{Proposition}\label{Phi}
The sequence $\Phi$ is a Leonard system on $M\hat{x}$ that has dual $q$-Krawtchouk type with parameter sequence $(\alpha,\alpha^*,\beta,\beta^*,\gamma\,;q,D)$, where\footnote{See also \cite[Example 6.1(6)]{PT.1993(2)}.}
\begin{gather*}
	\alpha = \frac{q^e-1}{1-q}, \qquad \beta = \frac{q^{D+e}}{q-1}, \qquad 	\gamma = \frac{1}{1-q}, \\
	\alpha^* = \frac{q(1+q^{D+e-2})}{1-q}, \qquad \beta^* = \frac{q(1+q^{D+e-2})(1+q^{D+e-1})}{(q-1)(1+q^{e-1})}.
\end{gather*}
Moreover, $\{A_i\hat{x}\}_{i=0}^D$ is a $\Phi$-standard basis for $M\hat{x}$.
\end{Proposition}

\begin{proof}
It follows from \eqref{3products;x} and \eqref{primary T(x)-module} that $\Phi$ is a Leonard system on $M\hat{x}$.
Moreover, $\{A_i\hat{x}\}_{i=0}^D$ is a $\Phi$-standard basis for $M\hat{x}$, since $A_i\hat{x}=E_i^*\hat{X}$, $0\leqslant i\leqslant D$, and $\hat{X}\in E_0M\hat{x}$.
By virtue of \eqref{Phi-std} and \eqref{recurrence 1}, we routinely obtain the parameter array of $\Phi$ from \eqref{int.num;x}, \eqref{e-val}, \eqref{dual-e-val} using \eqref{bi formula}, \eqref{ci formula} (or \eqref{bi, ci : dual-qK}), and f\/ind that $\Phi$ has dual $q$-Krawtchouk type as given above.
\end{proof}

We note that the intersection numbers of $\Phi$ coincide with those of $\Gamma$.

We next consider the $T$-module $M\hat{x}^{\perp}$, the orthogonal complement of $M\hat{x}$ in $\mathbf{W}$.
Observe that $E_0M\hat{x}^{\perp}=0$ since $E_0\mathbf{W}=\mathbb{C}\hat{X}\subset M\hat{x}$.
From \eqref{E_1} and \eqref{clique number} it follows that
\begin{gather}\label{inner products}
	||E_1\hat{x}||^2=\frac{\theta_0^*}{|X|}, \qquad \langle E_1\hat{x},E_1\hat{C}\rangle=\frac{1}{|X|}(\theta_0^*+q^e\theta_1^*), \qquad ||E_1\hat{C}||^2=\frac{1+q^e}{|X|}(\theta_0^*+q^e\theta_1^*),\!\!\!
\end{gather}
from which it follows that
\begin{gather*}
	w:=c\cdot\left(E_1\hat{C}-\frac{\theta_0^*+q^e\theta_1^*}{\theta_0^*}E_1\hat{x}\right) \in E_1M\hat{x}^{\perp},
\end{gather*}
where we take
\begin{gather*}
	c:=\frac{|X|(\alpha^*+\beta^*)}{\alpha^*\beta^*q^{e-1}(1-q)}
\end{gather*}
for convenience.
By \eqref{E_1}, \eqref{dual-e-val}, \eqref{E_1C}, and \eqref{dual-e-val;tilde}, we routinely obtain
\begin{gather*}
	w=\sum_{i=0}^{D-1}\big(\big(q^{-i}-1\big)\hat{C}^-_i+\big(q^{D-i-1}-1\big)\hat{C}^+_i\big).
\end{gather*}
Note that $E_0^*w=E_D^*w=0$.
Let
\begin{gather}\label{ui.perp}
	u_i^{\perp}:=E_{i+1}^*w=\big(q^{D-i-1}-1\big)\hat{C}^+_i + \big(q^{-i-1}-1\big)\hat{C}^-_{i+1}, \qquad 0\leqslant i\leqslant D-2.
\end{gather}
Since $\dim(M\hat{x}^{\perp})=D-1$, the vectors $\{u^{\perp}_i\}^{D-2}_{i=0}$ form an orthogonal basis for $M\hat{x}^{\perp}$.

Consider the sequence
\begin{gather*}
	\Phi^{\perp}:=\big(A;A^*;\{E_i\}^{D-1}_{i=1};\{E^*_i\}^{D-1}_{i=1}\big)\big|_{M\hat{x}^{\perp}}.
\end{gather*}

\begin{Proposition}\label{para.seq}
The $T$-module $M\hat{x}^{\perp}$ is irreducible, and $\Phi^{\perp}$ is a Leonard system on $M\hat{x}^{\perp}$ that has dual $q$-Krawtchouk type with parameter sequence $(\alpha^{\perp}, \alpha^{*\perp}, \beta^{\perp}, \beta^{*\perp}, \gamma^{\perp}\,;q,D-2)$, where\footnote{See also \cite[Theorem 6.9]{Tanaka2009LAAb}, \cite[Theorem 4.6]{PT.1993(1)}.}
\begin{gather*}
	\big(\alpha^{\perp}, \alpha^{*\perp}, \beta^{\perp}, \beta^{*\perp}, \gamma^{\perp}\big)
	= \big(\alpha, \alpha^*, \beta q^{-1}, \beta^* q^{-1}, \gamma q\big).
\end{gather*}
Moreover, $\{u^{\perp}_i\}^{D-2}_{i=0}$ is a $\Phi^{\perp}$-standard basis for $M\hat{x}^{\perp}$.
\end{Proposition}

\begin{proof}
By Lemma \ref{actA}, we have
\begin{gather}\label{A.vip}
	A.u^{\perp}_i = \frac{q^{D+e-1}-q^{i+e}}{q-1}u^{\perp}_{i-1} +\big(q^e-1\big)\frac{q^{i+1}-1}{q-1}u^{\perp}_i +\frac{q^{i+2}-q}{q-1}u^{\perp}_{i+1}, \qquad\!\! 0\leqslant i\leqslant D-2,\!\!\!
\end{gather}
where $u^{\perp}_{-1}=u^{\perp}_{D-1}:=0$. Note that the coef\/f\/icients of $u^{\perp}_{i-1}$, $u^{\perp}_{i+1}$ are non-zero.
Every $T$-submodule of $M\hat{x}^{\perp}$ is spanned by some of the $u_i^{\perp}$ since it is $M^*$-invariant, and~\eqref{A.vip} then shows that it must be either $M\hat{x}^{\perp}$ or $0$. Hence $M\hat{x}^{\perp}$ is an irreducible $T$-module. It also follows from~\eqref{A.vip} that $M\hat{x}^{\perp}=Mu_0^{\perp}$. In particular, $M\hat{x}^{\perp}$ is spanned by the vectors $E_iu_0^{\perp}$, $0\leqslant i\leqslant D$.
Note that $E_0u_0^{\perp}=0$ and $E_1u_0^{\perp}\ne 0$ since $E_0M\hat{x}^{\perp}=0$ and $w\in E_1M\hat{x}^{\perp}=\mathbb{C}E_1u_0^{\perp}$.
Hence it follows from~\eqref{3products;x} and the irreducibility of $M\hat{x}^{\perp}$ that $E_iu_0^{\perp}\ne 0$, $2\leqslant i\leqslant D-1$ and $E_Du_0^{\perp}=0$.
From these comments and \eqref{3products;x}, it follows that $\Phi^{\perp}$ is a Leonard system on $M\hat{x}^{\perp}$, and that $\{u_i^{\perp}\}_{i=0}^{D-2}$ is a $\Phi^{\perp}$-standard basis for $M\hat{x}^{\perp}$.
Note that
\begin{gather*}
	A^*.u^{\perp}_i=\theta_{i+1}^* u^{\perp}_i, \qquad 0\leqslant i\leqslant D-2.
\end{gather*}
The parameter array of $\Phi^{\perp}$ can now be computed as in the proof of Proposition~\ref{Phi}.
\end{proof}

We may remark that \emph{every} irreducible $T$-module indeed af\/fords a Leonard system of dual $q$-Krawtchouk type; cf.~\cite[Theorem 23.1]{CW}.

Let $\pi$ be the orthogonal projection from $\mathbf{W}$ onto $M\hat{x}$, i.e.,
\begin{gather*}
	(1-\pi)M\hat{x} = 0, \qquad \pi M\hat{x}^{\perp} = 0.
\end{gather*}
Then we have $A\pi=\pi A$ and $A^* \pi = \pi A^*$ on $\mathbf{W}$.
Observe that
\begin{gather}\label{C;v-com-1}
	\hat{C}^+_{i-1} = \frac{q^i-1}{q^D-1}A_i\hat{x} + \frac{q^i}{q^D-1} u^{\perp}_{i-1}, \qquad \hat{C}^-_i = \frac{q^D-q^i}{q^D-1}A_i\hat{x} + \frac{q^i}{1-q^D} u^{\perp}_{i-1}
\end{gather}
for $1\leqslant i \leqslant D-1$,
from which it follows that
\begin{gather}\label{act.pi}
	\pi.\hat{C}^+_{i-1} = \frac{q^i-1}{q^D-1}\big(\hat{C}^+_{i-1}+\hat{C}^-_i\big),
	\qquad
	\pi.\hat{C}^-_i = \frac{q^D-q^i}{q^D-1}\big(\hat{C}^+_{i-1}+\hat{C}^-_i\big)
\end{gather}
for $1 \leqslant i \leqslant D-1$. Moreover, we have $\pi.\hat{C}^-_0 = \hat{C}^-_0$ and $\pi.\hat{C}^+_{D-1} = \hat{C}^+_{D-1}$.

\section[Irreducible $\widetilde{T}$-submodules of $\mathbf{W}$]{Irreducible $\boldsymbol{\widetilde{T}}$-submodules of $\boldsymbol{\mathbf{W}}$}\label{S:T(C)-mod}

In this section, we focus on the decomposition \eqref{ods;Ttil} as $\widetilde{T}$-modules. We f\/irst discuss the primary $\widetilde{T}$-module $M\hat{C}$. Consider the sequence
\begin{gather*}
	\widetilde{\Phi} := \big(A;\widetilde{A}^*;\{E_i\}^{D-1}_{i=0};\{\widetilde{E}^*_i\}^{D-1}_{i=0}\big)\big|_{M\hat{C}}.
\end{gather*}
\begin{Proposition}\label{LS MC}
The sequence $\widetilde{\Phi}$ is a Leonard system on $M\hat{C}$ that has dual $q$-Krawtchouk type with parameter sequence $(\widetilde{\alpha}, \widetilde{\alpha}^*, \widetilde{\beta}, \widetilde{\beta}^*, \widetilde{\gamma}\,;q,D-1)$, where\footnote{See \cite[Proposition 4.6]{Tanaka2011EJC} for a more general result.} $($cf.~Proposition {\rm \ref{Phi})}
\begin{gather*}
	\big(\widetilde{\alpha}, \widetilde{\alpha}^*, \widetilde{\beta}, \widetilde{\beta}^*, \widetilde{\gamma}\big) = \left(\alpha, \alpha^*,\beta, \frac{1+q^{e-1}}{1+q^e}\beta^*, \gamma\right).
\end{gather*}
Moreover, $\{\hat{C}_i\}^{D-1}_{i=0}$ is a $\widetilde{\Phi}$-standard basis for $M\hat{C}$.
\end{Proposition}

\begin{proof}
We f\/irst remark that $E_D\hat{C}=0$. Indeed, the Hof\/fman bound $1-\theta_0\theta_D^{-1}$ on the size of a~clique follows from the fact that $E_D$ is positive semidef\/inite, and the bound is attained precisely when the characteristic vector of the clique vanishes on $E_DV$.
From this comment, \eqref{completely regular},~and~\eqref{3products;C}, it follows that $\tilde{\Phi}$ is a Leonard system on $M\hat{C}$.
Moreover, $\{\hat{C}_i\}_{i=0}^{D-1}$ is a $\widetilde{\Phi}$-standard basis for $M\hat{C}$, since $\hat{C}_i=\widetilde{E}_i^*\hat{X}$, $0\leqslant i\leqslant D-1$, and $\hat{X}\in E_0M\hat{C}$. In view of~\eqref{Phit-std}, the parameter array of $\widetilde{\Phi}$ can be computed from \eqref{e-val}, \eqref{int num C;q}, and \eqref{dual-e-val;tilde}, as in the proof of Proposition~\ref{Phi}.
\end{proof}

We note that the intersection numbers of $\widetilde{\Phi}$ coincide with those of $C$.

We next consider the $\widetilde{T}$-module $M\hat{C}^{\perp}$, the orthogonal complement of $M\hat{C}$ in $\mathbf{W}$.
From \eqref{inner products} it follows that
\begin{gather*}
	\widetilde{w}:=\widetilde{c}\cdot\left(E_1\hat{x}-\frac{1}{1+q^e}E_1\hat{C}\right) \in E_1M\hat{C}^{\perp},
\end{gather*}
where we take
\begin{gather*}
	\widetilde{c}:=\frac{|X|q(1+q^e)}{\beta^*(1-q)}
\end{gather*}
for convenience.
By \eqref{E_1}, \eqref{dual-e-val}, \eqref{E_1C}, and \eqref{dual-e-val;tilde}, we obtain
\begin{gather*}
	\widetilde{w}=\sum_{i=0}^{D-1}\big({-}q^{e-i}\hat{C}^-_i+q^{-i}\hat{C}^+_i\big).
\end{gather*}
Let
\begin{gather}\label{ui.perp.tilde}
	\widetilde{u}^{\perp}_i :=\widetilde{E}_i^*\widetilde{w}= -q^{e-i}\hat{C}^-_i + q^{-i}\hat{C}^+_i, \qquad 0\leqslant i\leqslant D-1.
\end{gather}
Since $\dim(M\hat{C}^{\perp})=D$, the vectors $\{\widetilde{u}^{\perp}_i\}^{D-1}_{i=0}$ form an orthogonal basis for $M\hat{C}^{\perp}$.

Consider the sequence
\begin{gather*}
	\widetilde{\Phi}^{\perp} := \big(A;\widetilde{A}^*;\{E_i\}^D_{i=1};\{\widetilde{E}^*_i\}^{D-1}_{i=0}\big)\big|_{M\hat{C}^{\perp}}.
\end{gather*}

\begin{Proposition}\label{para.seq.tilde}
The $\widetilde{T}$-module $M\hat{C}^{\perp}$ is irreducible, and $\widetilde{\Phi}^{\perp}$ is a Leonard system on $M\hat{C}^{\perp}$ that has dual $q$-Krawtchouk type with parameter sequence $(\widetilde{\alpha}^{\perp},\widetilde{\alpha}^{*\perp},\widetilde{\beta}^{\perp},\widetilde{\beta}^{*\perp},\widetilde{\gamma}^{\perp}\,;q,D-1)$, where
\begin{gather*}
	\big(\widetilde{\alpha}^{\perp}, \widetilde{\alpha}^{*\perp}, \widetilde{\beta}^{\perp}, \widetilde{\beta}^{*\perp}, \widetilde{\gamma}^{\perp}\big)
	= \left(\alpha, \alpha^*,\beta q^{-1}, \frac{1+q^{e-1}}{1+q^e}\beta^*, \gamma q\right).
\end{gather*}
Moreover, $\{\widetilde{u}^{\perp}_i\}^{D-1}_{i=0}$ is a $\widetilde{\Phi}^{\perp}$-standard basis for $M\hat{C}^{\perp}$.
\end{Proposition}

\begin{proof}
By Lemma \ref{actA}, we have
\begin{gather*}
	A.\widetilde{u}^{\perp}_i = \frac{q^{D+e-1}-q^{e+i-1}}{q-1}\widetilde{u}^{\perp}_{i-1} + \frac{q^{e+i}-q^e-q^{i+1}+1}{q-1} \widetilde{u}^{\perp}_i + \frac{q^{i+2}-q}{q-1}\widetilde{u}^{\perp}_{i+1}, \quad 0\leqslant i\leqslant D-1,
\end{gather*}
where $\widetilde{u}_{-1}^{\perp}=\widetilde{u}_D^{\perp}:=0$.
The result is proved using this identity as in the proof of Proposi\-tion~\ref{para.seq}.
\end{proof}

Let $\widetilde{\pi}$ be the orthogonal projection from $\mathbf{W}$ onto $M\hat{C}$, i.e.,
\begin{gather*}
	(1-\widetilde{\pi})M\hat{C} = 0, \qquad \widetilde{\pi} M\hat{C}^{\perp} = 0.
\end{gather*}
Then we have $A\widetilde{\pi}=\widetilde{\pi} A$ and $\widetilde{A}^* \widetilde{\pi} = \widetilde{\pi} \widetilde{A}^*$ on $\mathbf{W}$.
Observe that
\begin{gather}\label{C;v-com-2}
	\hat{C}^-_i = \frac{1}{1+q^e}\hat{C}_i - \frac{q^i}{1+q^e} \widetilde{u}^{\perp}_i, \qquad \hat{C}^+_i = \frac{q^e}{1+q^e}\hat{C}_i + \frac{q^i}{1+q^e} \widetilde{u}^{\perp}_i
\end{gather}
for $0\leqslant i \leqslant D-1$, from which it follows that
\begin{gather}\label{act.pitilde}
	\widetilde{\pi}.\hat{C}^-_i = \frac{1}{1+q^e}\big(\hat{C}^-_i+\hat{C}^+_i\big),
	\qquad
	\widetilde{\pi}.\hat{C}^+_i = \frac{q^e}{1+q^e}\big(\hat{C}^-_i+\hat{C}^+_i\big)
\end{gather}
for $0 \leqslant i \leqslant D-1$.

\section[A nil-DAHA of type $(C^{\vee}_1, C_1)$]{A nil-DAHA of type $\boldsymbol{(C^{\vee}_1, C_1)}$}\label{S:nil-DAHA}

In this section, we introduce a nil-DAHA of type $(C^{\vee}_1, C_1)$ and show that the primary $\mathbf{T}$-modu\-le~$\mathbf{W}$ also has a module structure for this algebra. Let $\kappa_0,\kappa_1,\kappa'_0,\kappa'_1\in\mathbb{C}$ be non-zero scalars. Recall that the \emph{DAHA} $\mathcal{H}=\mathcal{H}(\kappa_0,\kappa_1,\kappa'_0,\kappa'_1;q)$ \emph{of type} $(C^{\vee}_1, C_1)$ is generated by $\mathcal{T}_0^{\pm 1}$, $\mathcal{T}_1^{\pm 1}$, and~$\mathcal{X}^{\pm 1}$, subject to the relations (\cite[Section~6.4]{IM}, \cite[Section~3]{SS})
\begin{gather*}
	(\mathcal{T}_i-\kappa_i)\big(\mathcal{T}_i+\kappa^{-1}_i\big)=0, \qquad (\mathcal{T}'_i-\kappa'_i)\big(\mathcal{T}'_i+\kappa^{\prime -1}_i\big)=0, \qquad i=0,1,
\end{gather*}
where
\begin{gather*}
	\mathcal{T}_0':=q^{-1/2}\,\mathcal{X}\mathcal{T}_0^{-1}, \qquad \mathcal{T}_1':=\mathcal{X}^{-1}\mathcal{T}_1^{-1}.
\end{gather*}
Cherednik and Orr \cite{CO3} (cf.~\cite{CO1,CO2}) introduced the concept of \emph{nil-DAHAs} for reduced af\/f\/ine root systems.
The procedure for obtaining nil-DAHAs from ordinary DAHAs discussed in \cite[Section~2.5]{CO3} works for the non-reduced af\/f\/ine root systems of type $(C_n^{\vee},C_n)$ as well, with a bit of extra f\/lexibility in the specialization.\footnote{We learned this procedure for type $(C_n^{\vee},C_n)$ from Daniel Orr.}
It will turn out that the following specialization for type $(C^{\vee}_1, C_1)$ is the one which is well-suited to our situation:

\begin{Definition}\label{nil.H} Let $\kappa,\kappa' \in \mathbb{C}$ be non-zero scalars. Let $\overline{\mathcal{H}}=\overline{\mathcal{H}}(\kappa,\kappa')$ be the $\mathbb{C}$-algebra generated by~$\mathcal{T}^{\pm 1}$,~$\mathcal{U}$, and~$\mathcal{X}^{\pm 1}$, subject to the relations\footnote{The def\/inition of a nil-DAHA of type $(C_1^{\vee},C_1)$ given here is dif\/ferent from the one in \cite[Def\/inition~5.1]{FPSAC}. In fact, the former is a homomorphic image of the latter.}
\begin{gather*}
	(\mathcal{T} -\kappa)\big(\mathcal{T}+\kappa^{-1}\big)=0, \qquad (\mathcal{T}' -\kappa')\big(\mathcal{T}'+\kappa^{\prime -1}\big)=0, \qquad \mathcal{U}(\mathcal{U}+1)=0, \qquad \mathcal{U}^{\prime 2}=0,
\end{gather*}
where
\begin{gather*}
	\mathcal{T}'=\mathcal{X}\mathcal{T}^{-1}, \qquad \mathcal{U}'=\mathcal{X}^{-1}(\mathcal{U}+1).
\end{gather*}
We call $\overline{\mathcal{H}}$ a \emph{nil-DAHA of type} $(C^{\vee}_1, C_1)$.
\end{Definition}

\begin{Remark}
Our nil-DAHA $\overline{\mathcal{H}}$ is obtained from $\mathcal{H}$ as follows. Let $\ddot{\mathcal{T}}_1:=\kappa_1\mathcal{T}_1$, $\ddot{\mathcal{T}}'_1:=\kappa_1\mathcal{T}'_1$.
Then
\begin{gather}\label{ddot t'1}
	\ddot{\mathcal{T}}'_1=\kappa_1\mathcal{X}^{-1}\big(\mathcal{T}_1-\kappa_1+\kappa_1^{-1}\big)=\mathcal{X}^{-1}\big(\ddot{\mathcal{T}}_1-\kappa_1^2+1\big),
\end{gather}
and we have
\begin{gather}\label{Hecke in ddot}
	\big(\ddot{\mathcal{T}}_1-\kappa_1^2\big)\big(\ddot{\mathcal{T}}_1+1\big)=0, \qquad \big(\ddot{\mathcal{T}}'_1-\kappa_1\kappa_1'\big)\big(\ddot{\mathcal{T}}'_1+\kappa_1\kappa_1^{\prime-1}\big)=0.
\end{gather}
Take $\mathcal{T}_0^{\pm1},\mathcal{X}^{\pm1}$, and $\ddot{\mathcal{T}}_1$ as new generators for $\mathcal{H}$. We now set $\kappa_1'=1$ and let $\kappa_1\rightarrow 0$. Then~\eqref{ddot t'1} and~\eqref{Hecke in ddot} become
\begin{gather*}
	\ddot{\mathcal{T}}'_1=\mathcal{X}^{-1}\big(\ddot{\mathcal{T}}_1+1\big), \qquad \ddot{\mathcal{T}}_1\big(\ddot{\mathcal{T}}_1+1\big)=0, \qquad \big(\ddot{\mathcal{T}}'_1\big)^2=0,
\end{gather*}
and we obtain the presentation of $\overline{\mathcal{H}}$ in Def\/inition~\ref{nil.H} by the replacement
\begin{gather*}
\big(\mathcal{T}_0,q^{-1/2}\,\mathcal{X},\ddot{\mathcal{T}}_1,\kappa_0,\kappa'_0\big)\rightarrow(\mathcal{T},\mathcal{X},\mathcal{U},\kappa,\kappa').
\end{gather*}
We note that, if we instead set $\kappa_1'=\kappa_1$ and let $\kappa_1\rightarrow 0$, then the second identity in \eqref{Hecke in ddot} becomes
\begin{gather*}
	\ddot{\mathcal{T}}_1'\big(\ddot{\mathcal{T}}_1'+1\big)=0.
\end{gather*}
This gives another version of a nil-DAHA of type $(C_1^{\vee},C_1)$, which we expect would be suitable for the Grassmann graphs \cite[Section~9.3]{BCN} corresponding to the dual $q$-Hahn polynomials \cite[Section~14.7]{KLS2010B}. We may also apply the above procedure to~$\mathcal{T}_0$ as well to get more variations.
\end{Remark}

\begin{Remark}By either or both of the replacements $\mathcal{T}\rightarrow-\mathcal{T}$, $\mathcal{X}\rightarrow-\mathcal{X}$, it is immediate to see that the four nil-DAHAs $\overline{\mathcal{H}}(\pm\kappa,\pm\kappa')$ are all isomorphic.
\end{Remark}

\begin{Remark}\label{Mazzocco's work 1}Mazzocco \cite{Mazzocco2016N} def\/ined seven new algebras as conf\/luences of the DAHA $\mathcal{H}$, and established a new link between the theory of the Painlev\'{e} equations and part of the left side of the $q$-Askey scheme. One of these algebras, called $\mathcal{H}_{\mathrm{III}}$, turns out to be isomorphic to the $\mathbb{C}$-algebra having the same generators and relations as the nil-DAHA $\overline{\mathcal{H}}$, plus one further relation $\mathcal{U}'=q\mathcal{U}\mathcal{X}$; cf.~\cite[equations~(3.86)--(3.91)]{Mazzocco2016N}.
\end{Remark}

For the rest of this paper, we set
\begin{gather}\label{kappa}
	\kappa=q^{-e/2}, \qquad \kappa' = \sqrt{-1} q^{-D/2}
\end{gather}
in Def\/inition \ref{nil.H}, where we recall that $q$ is assumed to be a prime power in our context. Our f\/irst goal is to def\/ine a $2D$-dimensional representation of $\overline{\mathcal{H}}=\overline{\mathcal{H}}(\kappa,\kappa')$. To this end, we consider the following matrices:
\begin{itemize}\itemsep=0pt
\item for $0 \leqslant i \leqslant D-1$, let
\begin{gather*}
	t(i) = \begin{pmatrix}
			q^{-e/2}-q^{e/2} & q^{e/2} \\
			q^{-e/2} & 0
		\end{pmatrix}, \qquad
	u'(i) = \begin{pmatrix}
			0 & 0\\
			-\sqrt{-1}\,q^{(D-e)/2-i} & 0
		\end{pmatrix};
\end{gather*}
\item for $1 \leqslant i \leqslant D-1$, let
\begin{gather*}
	t'(i) = \sqrt{-1}\begin{pmatrix}
			q^{-D/2}(q^D-q^i+1) & q^{D/2}(q^{i-D}-1)\\
			q^{-D/2}(1-q^i) & q^{i-D/2}
		\end{pmatrix}, \qquad
	u(i) = \begin{pmatrix}
			-1 & 1-q^{D-i} \\
			0 & 0
		\end{pmatrix},
\end{gather*}
and let
$t'(0) = \begin{pmatrix} \sqrt{-1}\,q^{-D/2} \end{pmatrix}$,
$t'(D) = \begin{pmatrix} \sqrt{-1}\,q^{-D/2} \end{pmatrix}$,
$u(0) = \begin{pmatrix} 0 \end{pmatrix}$, and
$u(D) = \begin{pmatrix} -1 \end{pmatrix}$.
\end{itemize}

\begin{Lemma}\label{t+t^(-1)}The following $(i)$, $(ii)$ hold:
\begin{enumerate}\itemsep=0pt
\item[$(i)$] $(t(i)-\kappa)(t(i)+\kappa^{-1})=0$ for $0\leqslant i\leqslant D-1$.
\item[$(ii)$] $(t'(i)-\kappa')(t'(i)+\kappa^{\prime -1})=0$ for $0\leqslant i\leqslant D$.
\end{enumerate}
\end{Lemma}

\begin{proof}(i) Immediate from $\det(t(i))=-1$ and $\operatorname{trace}(t(i))=\kappa-\kappa^{-1}$.

(ii) Similar to the proof of (i) above.
\end{proof}

\begin{Lemma}\label{u}
The following $(i)$, $(ii)$ hold:
\begin{enumerate}\itemsep=0pt
\item[$(i)$] $u(i)(u(i)+1)=0$ for $0\leqslant i\leqslant D$.
\item[$(ii)$] $(u'(i))^2=0$ for $0\leqslant i\leqslant D-1$.
\end{enumerate}
\end{Lemma}

\begin{proof}
Routine.
\end{proof}

Def\/ine the $2D \times 2D$ block diagonal matrices $\mathfrak{t}$, $\mathfrak{t}'$, $\mathfrak{u}$, and $\mathfrak{u}'$ by
\begin{gather*}
	\mathfrak{t} = \operatorname{blockdiag}\bigl( t(0), t(1), \dots, t(D-1) \bigr), \\
	\mathfrak{t}' = \operatorname{blockdiag}\bigl( t'(0), t'(1), \dots, t'(D-1), t'(D) \bigr), \\
	\mathfrak{u} = \operatorname{blockdiag}\bigl( u(0), u(1), \dots, u(D-1), u(D) \bigr), \\
	\mathfrak{u}' = \operatorname{blockdiag}\bigl( u'(0), u'(1), \dots, u'(D-1) \bigr).
\end{gather*}
Moreover, let
\begin{gather*}
	\mathfrak{x} = \mathfrak{t}'\mathfrak{t}.
\end{gather*}

\begin{Proposition}\label{representation}
There is a representation $\overline{\mathcal{H}}\rightarrow \operatorname{Mat}_{2D}(\mathbb{C})$ such that
\begin{gather*}
	\mathcal{T}\mapsto\mathfrak{t}, \qquad \mathcal{T}'\mapsto\mathfrak{t}', \qquad \mathcal{U}\mapsto\mathfrak{u}, \qquad \mathcal{U}'\mapsto\mathfrak{u}', \qquad \mathcal{X}\mapsto\mathfrak{x}.
\end{gather*}
\end{Proposition}

\begin{proof}From Lemmas \ref{t+t^(-1)} and \ref{u} it follows that
\begin{gather*}
	(\mathfrak{t} -\kappa)\big(\mathfrak{t}+\kappa^{-1}\big)=(\mathfrak{t}' -\kappa')\big(\mathfrak{t}'+\kappa^{\prime -1}\big)=\mathfrak{u}(\mathfrak{u}+1)=\mathfrak{u}^{\prime 2}=0.
\end{gather*}
In particular, $\mathfrak{t}$, $\mathfrak{t}'$, and $\mathfrak{x}$ are invertible. We have $\mathfrak{t}'=\mathfrak{x}\mathfrak{t}^{-1}$ by def\/inition, and it is a straightforward matter to show that $\mathfrak{x}\mathfrak{u}'=\mathfrak{u}+1$. It follows that $\mathfrak{t}^{\pm 1}$, $\mathfrak{u}$, and $\mathfrak{x}^{\pm 1}$ satisfy the def\/ining relations for~$\overline{\mathcal{H}}$, and the result follows.
\end{proof}

\begin{Corollary}\label{H-mod}There is an $\overline{\mathcal{H}}$-module structure on $\mathbf{W}$ such that $\mathfrak{t}$, $\mathfrak{t}'$, $\mathfrak{u}$, $\mathfrak{u}'$, and~$\mathfrak{x}$ are respectively the matrices representing the actions of $\mathcal{T}$, $\mathcal{T}'$, $\mathcal{U}$, $\mathcal{U}'$, and~$\mathcal{X}$ in the ordered basis~$\mathcal{C}$ from~\eqref{1st basis}.
\end{Corollary}

By Corollary \ref{H-mod}, $\mathbf{W}$ is now a module for both $\mathbf{T}$ and $\overline{\mathcal{H}}=\overline{\mathcal{H}}(\kappa,\kappa')$, where $\kappa$, $\kappa'$ are given as in \eqref{kappa}. We next discuss how the two module structures are related. Let
\begin{gather*}
	\mathcal{A} =\mathcal{X}+\mathcal{X}^{-1}, \qquad \mathcal{A}^* = \sqrt{-1}\,q^{-D/2}(\mathcal{T}\mathcal{U}'+\mathcal{U}\mathcal{T}'), \qquad \widetilde{\mathcal{A}}^* = \sqrt{-1}\,q^{-D/2}(\mathcal{T}\mathcal{U}'+q\,\mathcal{U}\mathcal{T}').
\end{gather*}
For the rest of this paper, we also let
\begin{gather}\label{dK parameter}
	\tau=\sqrt{-1}\,q^{-(D+e)/2}.
\end{gather}
Note that $\tau^2=\beta^{-1}\gamma$, where $\beta$, $\gamma$ are from Proposition~\ref{Phi}; cf.~\eqref{tau}.

The following four lemmas are checked by straightforward calculations. Recall \eqref{empty in some cases}.

\begin{Lemma}\label{Y.C-i}For $0 \leqslant i \leqslant D-1$, the actions of $\mathcal{X}^{\pm 1}$ on $\hat{C}^-_i$ are given as linear combinations with the following terms and coefficients:
\begin{gather*} \renewcommand\arraystretch{1.3}
	\mathcal{X}.\hat{C}^-_i\colon \
	\begin{array}{l | c}
	\text{\textup{term}} & \text{\textup{coef\/f\/icient}} \\
	\hline \hline
	\hat{C}^+_{i-1} & (\tau q^D+\tau^{-1})(q^{i-D}-1)\\
	\hat{C}^-_i & (\tau q^D+\tau^{-1})q^{i-D}\\
	\hat{C}^+_i & \tau(q^D-q^{i+1}+1)\\
	\hat{C}^-_{i+1} & \tau(1-q^{i+1})
	\end{array}
	\qquad
	\mathcal{X}^{-1}.\hat{C}^-_i\colon \
	\begin{array}{l | c}
	\text{\textup{term}} & \text{\textup{coef\/f\/icient}} \\
	\hline \hline
	\hat{C}^-_{i-1} & \tau^{-1}(1-q^{i-D}) \\
	\hat{C}^+_{i-1} & (\tau q^D+\tau^{-1})(1-q^{i-D})\\
	\hat{C}^-_i & 0\\
	\hat{C}^+_i & -\tau(q^D-q^i+1)
	\end{array}
\end{gather*}
\end{Lemma}

\begin{Lemma}\label{Y.C+i}
For $0 \leqslant i \leqslant D-1$, the actions of $\mathcal{X}^{\pm 1}$ on $\hat{C}^+_i$ are given as linear combinations with the following terms and coefficients:
\begin{gather*}
	\renewcommand\arraystretch{1.3}
	\mathcal{X}.\hat{C}^+_i\colon \	
	\begin{array}{l | c}
	\text{\textup{term}} & \text{\textup{coef\/f\/icient}} \\
	\hline \hline
	\hat{C}^+_{i-1} & \tau^{-1}(1-q^{i-D})\\
	\hat{C}^-_i & -\tau^{-1}q^{i-D}\\
	\hat{C}^+_i & 0\\
	\hat{C}^-_{i+1} & 0
	\end{array}
	\qquad
	\mathcal{X}^{-1}.\hat{C}^+_i\colon \
	\begin{array}{l | c}
	\text{\textup{term}} & \text{\textup{coef\/f\/icient}} \\
	\hline \hline
	\hat{C}^-_i & \tau^{-1}q^{i-D+1} \\
	\hat{C}^+_i & (\tau q^D+\tau^{-1})q^{i-D+1}\\
	\hat{C}^-_{i+1} & 0\\
	\hat{C}^+_{i+1} & \tau(1-q^{i+1})
	\end{array}
\end{gather*}
\end{Lemma}

\begin{Lemma}\label{A.C+-}
For $0 \leqslant i \leqslant D-1$, the actions of $\mathcal{A}$ on $\hat{C}^{\pm}_i$ are given as linear combinations with the following terms and coefficients:
	\begin{gather*}\renewcommand\arraystretch{1.3}
	\mathcal{A}.\hat{C}^-_i\colon \
	\begin{array}{l | c}
	\text{\textup{term}} & \text{\textup{coef\/f\/icient}} \\
	\hline \hline
	\hat{C}^-_{i-1} & \tau^{-1}(1-q^{i-D})\\
	\hat{C}^+_{i-1} & 0\\
	\hat{C}^-_i & (\tau q^D+\tau^{-1})q^{i-D}\\
	\hat{C}^+_i & \tau q^i(1-q)\\
	\hat{C}^-_{i+1} & \tau(1-q^{i+1})
	\end{array}
	\qquad
	\mathcal{A}.\hat{C}^+_i\colon \
	\begin{array}{l | c}
	\text{\textup{term}} & \text{\textup{coef\/f\/icient}} \\
	\hline \hline
	\hat{C}^+_{i-1} & \tau^{-1}(1-q^{i-D})\\
	\hat{C}^-_i & \tau^{-1}q^{i-D}(q-1)\\
	\hat{C}^+_i & (\tau q^D+\tau^{-1})q^{i-D+1}\\
	\hat{C}^-_{i+1} & 0\\
	\hat{C}^+_{i+1} & \tau(1-q^{i+1})
	\end{array}
	\end{gather*}
\end{Lemma}

\begin{Lemma}\label{B.C+-}For $0 \leqslant i \leqslant D-1$, the actions of $\mathcal{A}^*$, $\widetilde{\mathcal{A}}^*$ on $\hat{C}_i^{\pm}$ are given by
\begin{gather*}
	\mathcal{A}^*.\hat{C}^{-}_i = q^{-i}\hat{C}^-_i, \qquad
	\mathcal{A}^*.\hat{C}^{+}_i = q^{-i-1}\hat{C}^+_i, \\
	\widetilde{\mathcal{A}}^*.\hat{C}^{-}_i = q^{-i}\hat{C}^-_i, \qquad
	\widetilde{\mathcal{A}}^*.\hat{C}^{+}_i = q^{-i}\hat{C}^+_i.
\end{gather*}
\end{Lemma}

Recall the generators $A$, $A^*$, and $\widetilde{A}^*$ of $\mathbf{T}$. We now present our f\/irst main result.

\begin{Theorem} \label{thm1}
On $\mathbf{W}$, we have
\begin{gather*}
	A = \alpha+\beta\tau\mathcal{A}, \qquad A^* = \alpha^*+\beta^*\mathcal{A}^*, \qquad \widetilde{A}^* = \widetilde{\alpha}^*+\widetilde{\beta}^*\widetilde{\mathcal{A}}^*,
\end{gather*}
where $\alpha$, $\alpha^*$, $\widetilde{\alpha}^*$, $\beta$, $\beta^*$, $\widetilde{\beta}^*$ are from Propositions~{\rm \ref{Phi}} and {\rm \ref{LS MC}}.
Moreover, $\mathbf{W}$ is an irreducible $\overline{\mathcal{H}}$-module.
\end{Theorem}

\begin{proof}The three identities follow from Lemmas~\ref{actA}, \ref{act.A*,A*t}, \ref{A.C+-}, \ref{B.C+-}, and Propositions~\ref{Phi},~\ref{LS MC}.
Since $\mathbf{W}$ is an irreducible $\mathbf{T}$-module by Proposition \ref{W:irrT-mod} and since $A$, $A^*$, and $\widetilde{A}^*$ generate $\mathbf{T}$, it follows from these identities that $\mathbf{W}$ is irreducible as an $\overline{\mathcal{H}}$-module.
\end{proof}

Recall the orthogonal projection $\pi$ (resp.~$\widetilde{\pi}$) from $\mathbf{W}$ onto $M\hat{x}$ (resp.~$M\hat{C}$). The following result illustrates (to some extent) how we arrived at the $\overline{\mathcal{H}}$-module structure on $\mathbf{W}$ given above:

\begin{Theorem}\label{thm2} On $\mathbf{W}$, we have
\begin{gather*}
	\pi = \frac{\mathcal{T}'+\kappa^{\prime-1}}{\kappa'+\kappa^{\prime-1}}, \qquad \widetilde{\pi} = \frac{\mathcal{T}+\kappa^{-1}}{\kappa+\kappa^{-1}}.
\end{gather*}
\end{Theorem}
\begin{proof}Use \eqref{act.pi}, \eqref{act.pitilde}.
\end{proof}

\begin{Remark}\label{Mazzocco's work 2} A direct computation shows that we also have $\mathfrak{u}'=q\mathfrak{u}\mathfrak{x}$ in Proposition~\ref{representation}, so that, by Theorem~\ref{thm1}, $\mathbf{W}$~is indeed an irreducible module for the algebra $\mathcal{H}_{\mathrm{III}}$ mentioned in Remark~\ref{Mazzocco's work 1}. The f\/inite-dimensional irreducible modules for $\mathcal{H}$ have been classif\/ied by Oblomkov and Stoica~\cite{OS2009JPAA} in the general case where the scalar $q$ is not a root of unity. See also~\cite{NT2017LAA}. It would be an interesting problem to specialize the classif\/ication to~$\overline{\mathcal{H}}$ and/or~$\mathcal{H}_{\mathrm{III}}$. We may remark that Mazzocco \cite[Section~3]{Mazzocco2014SIGMA} obtained (among other results) a faithful representation of~$\mathcal{H}_{\mathrm{III}}$ on~$\mathbb{C}[\eta,\eta^{-1}]$. She introduced the \emph{non-symmetric Al-Salam--Chihara polynomials} in proving the faithfulness of this representation, and the non-symmetric dual $q$-Krawtchouk polynomials~$\ell_i^{\pm}$ which we will discuss in Sections~\ref{S:nsdqKpoly} and~\ref{S:orth rels} are a discretization of these Laurent polynomials.
\end{Remark}

\section[Non-symmetric dual $q$-Krawtchouk polynomials]{Non-symmetric dual $\boldsymbol{q}$-Krawtchouk polynomials}\label{S:nsdqKpoly}

Recall the four Leonard systems $\Phi$, $\Phi^{\perp}$, $\widetilde{\Phi}$, and $\widetilde{\Phi}^{\perp}$, and their standard bases $\{A_i\hat{x}\}_{i=0}^D$, $\{u^{\perp}_i\}^{D-2}_{i=0}$, $\{\hat{C}_i\}^{D-1}_{i=0}$, and $\{\widetilde{u}^{\perp}_i\}^{D-1}_{i=0}$ from Propositions \ref{Phi}, \ref{para.seq}, \ref{LS MC}, and \ref{para.seq.tilde}, respectively. We consider the monic dual $q$-Krawtchouk (Laurent) polynomials from \eqref{monic symmetric Laurent dual q-Krawtchouk} attached to these Leonard systems.
More specif\/ically, with the notation~\eqref{hi in 3 parameters} we let
\begin{alignat*}{3}
& h_i(\eta)=h_i(\eta;\tau,D;q),\qquad && 0\leqslant i \leqslant D,& \\
& h_i^{\perp}(\eta)=h_i(\eta;\tau q,D-2;q),\qquad && 0\leqslant i \leqslant D-2,& \\
&\widetilde{h}_i(\eta)=h_i(\eta;\tau,D-1;q),\qquad && 0\leqslant i \leqslant D-1, & \\
& \widetilde{h}_i^{\perp}(\eta)=h_i(\eta;\tau q,D-1;q),\qquad && 0\leqslant i \leqslant D-1 &
\end{alignat*}
for $\Phi$, $\Phi^{\perp}$, $\widetilde{\Phi}$, and $\widetilde{\Phi}^{\perp}$, respectively, where the scalar $\tau$ is from \eqref{dK parameter}.
Indeed, we have $\beta^{-1}\gamma=\widetilde{\beta}^{-1}\widetilde{\gamma}=\tau^2$ and $(\beta^{\perp})^{-1}\gamma^{\perp}=(\widetilde{\beta}^{\perp})^{-1}\widetilde{\gamma}^{\perp}=(\tau q)^2$.
We moreover set
\begin{gather*}
	p^{\perp}(\eta) = \eta^{-1}(\eta-\tau)\big(\eta-\tau^{-1}q^{-D}\big), \\
	\widetilde{p}(\eta)= \eta^{-1}\big(\eta-\tau^{-1}q^{-D}\big), \\
	\widetilde{p}^{\perp}(\eta) = \eta^{-1}(\eta-\tau).
\end{gather*}
Note that $p^{\perp}$, $\widetilde{p}$, and $\widetilde{p}^{\perp}$ are monic. Using these Laurent polynomials, we now def\/ine
\begin{align*}
	\ell^+_i(\eta) &= \frac{h_{i+1}-p^{\perp}h^{\perp}_i}{\tau^{i+1}(1-q^D)(q;q)_i}, & \ell^-_i(\eta) &= \frac{q^D-q^i}{\tau^i(q^D-1)(q;q)_i}
			\left(
			h_i-\frac{1-q^i}{q^D-q^i} p^{\perp}h^{\perp}_{i-1}
			\right), \\
	\widetilde{\ell}^+_i(\eta) &= \frac{q^e\bigl(\widetilde{p}\,\widetilde{h}_i-\widetilde{p}^{\perp}\widetilde{h}^{\perp}_i\bigr)}{\tau^i(1+q^e)(q;q)_i}, & \widetilde{\ell}^-_i(\eta) &= \frac{\widetilde{p}\,\widetilde{h}_i+q^e\widetilde{p}^{\perp}\widetilde{h}^{\perp}_i}{\tau^i(1+q^e)(q;q)_i}
\end{align*}
for $0\leqslant i\leqslant D-1$, where $h^{\perp}_{-1}:=0$ and
\begin{gather*}
	h^{\perp}_{D-1}(\eta):=\prod_{n=1}^{D-1}\big(\eta+\eta^{-1}-\tau q^n-\tau^{-1} q^{-n}\big)=\eta^{1-D}\prod_{n=1}^{D-1}\big(\eta-\tau q^n\big)\big(\eta-\tau^{-1}q^{-n}\big).
\end{gather*}
Let
\begin{gather}\label{L}
	\mathfrak{L}=\sum_{i=-D}^{D-1}\!\!\mathbb{C}\,\eta^i\subset\mathbb{C}\big[\eta,\eta^{-1}\big].
\end{gather}
Observe that $\dim(\mathfrak{L})=2D$, and that the $\ell^{\pm}_i$ and the $\widetilde{\ell}^{\pm}_i$ all belong to $\mathfrak{L}$. Our def\/inition of these Laurent polynomials is explained by the following result (and its proof):

\begin{Proposition}\label{C;K.v0}
For $0 \leqslant i \leqslant D-1$, we have
\begin{gather}\label{actions of li+-}
	\hat{C}^+_i = \ell^+_i(\mathcal{X}).\hat{x}=\widetilde{\ell}^+_i(\mathcal{X}).\hat{x}, \qquad \hat{C}^-_i = \ell^-_i(\mathcal{X}).\hat{x}=\widetilde{\ell}^-_i(\mathcal{X}).\hat{x}.
\end{gather}
In particular, we have $\ell^+_i=\widetilde{\ell}^+_i$, $\ell^-_i=\widetilde{\ell}^-_i$ for $0\leqslant i\leqslant D-1$. Moreover, the $\ell^{\pm}_i$ form a basis for~$\mathfrak{L}$.
\end{Proposition}

\begin{proof}
Setting $i=0$ in Lemma \ref{Y.C-i}, we have
\begin{gather*}
	\mathcal{X}.\hat{x} = \big(\tau+\tau^{-1}q^{-D}\big)\hat{x} + \tau\big(q^D-q+1\big)\hat{C}^+_0 + \tau(1-q)\hat{C}^-_1, \qquad	\mathcal{X}^{-1}.\hat{x} = -\tau q^D\hat{C}^+_0.
\end{gather*}
By these identities, \eqref{ui.perp}, and \eqref{ui.perp.tilde}, we have
\begin{gather}\label{0th std vectors}
	u_0^{\perp}=\tau^{-1}q^{-1}p^{\perp}(\mathcal{X}).\hat{x}, \qquad \hat{C}_0=\widetilde{p}(\mathcal{X}).\hat{x}, \qquad \widetilde{u}_0^{\perp}=-q^e\,\widetilde{p}^{\perp}(\mathcal{X}).\hat{x}.
\end{gather}

We f\/irst show that $\hat{C}^+_i = \ell^+_i(\mathcal{X}).\hat{x}$. By virtue of Propositions~\ref{Phi},~\ref{para.seq}, and Theorem~\ref{thm1}, we may apply the discussions in the last paragraph of Section~\ref{S:pre LS} to the Leonard systems~$\Phi$,~$\Phi^{\perp}$. Assume that $0\leqslant i\leqslant D-2$. Then, from~\eqref{hi on std basis} and~\eqref{C;v-com-1} it follows that
\begin{gather*}
	\hat{C}^+_i=\frac{q^{i+1}-1}{q^D-1}\frac{h_{i+1}(\mathcal{X})}{\tau^{i+1}(q;q)_{i+1}}.\hat{x} + \frac{q^{i+1}}{q^D-1}\frac{h^{\perp}_i(\mathcal{X})}{\tau^iq^i(q;q)_i}.u^{\perp}_0,
\end{gather*}
which, together with \eqref{0th std vectors}, gives $\hat{C}^+_i = \ell^+_i(\mathcal{X}).\hat{x}$. When $i=D-1$, by \eqref{hi on std basis} we have
\begin{gather*}
	\hat{C}_{D-1}^+=A_D\hat{x}=\frac{h_D(\mathcal{X})}{\tau^D(q;q)_D}.\hat{x}.
\end{gather*}
On the other hand, from Proposition \ref{Phi} and Theorem \ref{thm1}, it follows that
\begin{gather*}
	h_{D-1}^{\perp}(\mathcal{X})=\beta^{1-D}\tau^{1-D}\prod_{n=1}^{D-1}(A-\theta_n)
\end{gather*}
on $\mathbf{W}$.
Hence it follows from Proposition \ref{para.seq} that $h_{D-1}^{\perp}(\mathcal{X})$ vanishes on $M\hat{x}^{\perp}$, so that by \eqref{0th std vectors} we have in particular
\begin{gather}\label{vanishing}
	h_{D-1}^{\perp}(\mathcal{X})p^{\perp}(\mathcal{X}).\hat{x}=\tau q\,h_{D-1}^{\perp}(\mathcal{X}).u^{\perp}_0=0.
\end{gather}
Combining these comments, we obtain $\hat{C}^+_{D-1} = \ell^+_{D-1}(\mathcal{X}).\hat{x}$. The other identities in~\eqref{actions of li+-} are similarly (and more easily) proved, using also Propositions~\ref{LS MC},~\ref{para.seq.tilde}, and~\eqref{C;v-com-2}.

Finally, from \eqref{actions of li+-} it follows that the $\ell^{\pm}_i$ and the $\widetilde{\ell}^{\pm}_i$ form two bases for $\mathfrak{L}$, but then~\eqref{actions of li+-} again shows that we must have $\ell^+_i=\widetilde{\ell}^+_i$, $\ell^-_i=\widetilde{\ell}^-_i$ for $0\leqslant i\leqslant D-1$.
This completes the proof.
\end{proof}

\begin{Definition}\label{nDK}We call the Laurent polynomials $\ell^{\pm}_i=\widetilde{\ell}^{\pm}_i$, $0\leqslant i\leqslant D-1$, the \emph{non-symmetric dual $q$-Krawtchouk polynomials}.
\end{Definition}

\begin{Remark}
We have given two expressions for the non-symmetric dual $q$-Krawtchouk polynomials. The f\/irst one, i.e., the $\ell^{\pm}_i$, is a specialization of the expressions for the non-symmetric $q$-Racah polynomials\footnote{Strictly speaking, this is true except for $\ell^+_{D-1}$, in which case we adjusted it so that it has no term of degree~$D$. A similar adjustment should also be possible for the non-symmetric $q$-Racah polynomials.} in \cite[Proposition~7.8]{JHL2} and for the non-symmetric Askey--Wilson polynomials in \cite[Section~4]{KB2011AA} (but with a dif\/ferent normalization; cf.~\cite[Section~11]{JHL2}). It is a natural guess that the polynomials studied in~\cite{KB2011AA, JHL2} also have expressions akin to the second one, i.e., the $\widetilde{\ell}^{\pm}_i$.
\end{Remark}

We next show that the $\ell_i^{\pm}$ satisfy recurrence relations with at most four terms.
As a by-product of the proof of Proposition \ref{C;K.v0} it follows that

\begin{Lemma}\label{giving minimal polynomial}
On $\mathbf{W}$, we have
\begin{gather*}
	p^{\perp}(\mathcal{X})h_{D-1}^{\perp}(\mathcal{X})=0.
\end{gather*}
\end{Lemma}

\begin{proof}
From Proposition \ref{C;K.v0} and \eqref{vanishing} it follows that
\begin{gather*}
	p^{\perp}(\mathcal{X})h_{D-1}^{\perp}(\mathcal{X}).\hat{C}_i^{\pm}=\ell_i^{\pm}(\mathcal{X})h_{D-1}^{\perp}(\mathcal{X})p^{\perp}(\mathcal{X}).\hat{x}=0
\end{gather*}
for $0\leqslant i\leqslant D-1$.
\end{proof}

\noindent
For convenience, let
\begin{gather*}
	\ell_{-1}^{\pm}(\eta):=0, \qquad \ell_D^+(\eta):=-\frac{\eta^{-1}p^{\perp}h_{D-1}^{\perp}}{\tau^{D+1}(q;q)_D}, \qquad \ell_D^-(\eta):=\frac{p^{\perp}h_{D-1}^{\perp}}{\tau^D(q;q)_D}.
\end{gather*}
Observe that
\begin{gather}\label{mod}
	\eta^{-1}p^{\perp}h_{D-1}^{\perp}\equiv q^{-D}\eta^{-D-1} \ \ (\mathrm{mod}\, \mathfrak{L}), \qquad p^{\perp}h_{D-1}^{\perp}\equiv \eta^D \ \ (\mathrm{mod}\, \mathfrak{L}).
\end{gather}

\begin{Theorem}\label{thm4}
The following $(i)$, $(ii)$ hold:
\begin{enumerate}\itemsep=0pt
\item[$(i)$] For $0 \leqslant i \leqslant D-1$, $\eta^{\pm 1}\ell_i^-$ are linear combinations with the following terms and coefficients:
\begin{gather*} \renewcommand\arraystretch{1.3}
	\eta\ell_i^-\colon \
	\begin{array}{l | c}
	\text{\textup{term}} & \text{\textup{coef\/f\/icient}} \\
	\hline \hline
	\ell^+_{i-1} & (\tau q^D+\tau^{-1})(q^{i-D}-1)\\
	\ell^-_i & (\tau q^D+\tau^{-1})q^{i-D}\\
	\ell^+_i & \tau(q^D-q^{i+1}+1)\\
	\ell^-_{i+1} & \tau(1-q^{i+1})
	\end{array}
	\qquad
	\eta^{-1}\ell_i^-\colon \
	\begin{array}{l | c}
	\text{\textup{term}} & \text{\textup{coef\/f\/icient}} \\
	\hline \hline
	\ell^-_{i-1} & \tau^{-1}(1-q^{i-D}) \\
	\ell^+_{i-1} & (\tau q^D+\tau^{-1})(1-q^{i-D})\\
	\ell^-_i & 0\\
	\ell^+_i & -\tau(q^D-q^i+1)
	\end{array}
\end{gather*}
\item[$(ii)$] For $0 \leqslant i \leqslant D-1$, $\eta^{\pm 1}\ell_i^+$ are linear combinations with the following terms and coefficients:
\begin{gather*}
	\renewcommand\arraystretch{1.3}
	\eta\ell_i^+\colon \	
	\begin{array}{l | c}
	\text{\textup{term}} & \text{\textup{coef\/f\/icient}} \\
	\hline \hline
	\ell^+_{i-1} & \tau^{-1}(1-q^{i-D})\\
	\ell^-_i & -\tau^{-1}q^{i-D}\\
	\ell^+_i & 0\\
	\ell^-_{i+1} & 0
	\end{array}
	\qquad
	\eta^{-1}\ell_i^+\colon \
	\begin{array}{l | c}
	\text{\textup{term}} & \text{\textup{coef\/f\/icient}} \\
	\hline \hline
	\ell^-_i & \tau^{-1}q^{i-D+1} \\
	\ell^+_i & (\tau q^D+\tau^{-1})q^{i-D+1}\\
	\ell^-_{i+1} & 0\\
	\ell^+_{i+1} & \tau(1-q^{i+1})
	\end{array}
\end{gather*}
\end{enumerate}
\end{Theorem}

\begin{proof}
Looking at the two expressions $\ell_i^{\pm}=\widetilde{\ell}_i^{\pm}$ it follows that $\ell_i^+$ has highest degree (at most) $i-1$ and lowest degree $-i-1$, whereas $\ell_i^-$ has highest degree $i$ and lowest degree $-i$.
Hence, except for $\eta^{-1}\ell_{D-1}^+$, $\eta\ell_{D-1}^-$ we have $\eta^{\pm 1}\ell_i^{\pm}\in\mathfrak{L}$, and the result follows from Lemmas \ref{Y.C-i}, \ref{Y.C+i} and Proposition \ref{C;K.v0}.
For the remaining two cases, using \eqref{mod} we routinely f\/ind that
\begin{gather*}
	\eta^{-1}\ell_{D-1}^+ \equiv \tau\big(1-q^D\big)\ell_D^+ \ \ (\mathrm{mod}\, \mathfrak{L}), \qquad \eta\ell_{D-1}^- \equiv \tau\big(1-q^D\big)\ell_D^- \ \ (\mathrm{mod}\, \mathfrak{L}).
\end{gather*}
Since $\ell_D^{\pm}(\mathcal{X})$ vanish on $\mathbf{W}$ by virtue of Lemma \ref{giving minimal polynomial}, the result again follows from Lem\-mas~\ref{Y.C-i},~\ref{Y.C+i} and Proposition \ref{C;K.v0}.
\end{proof}

\section{Orthogonality relations}\label{S:orth rels}

Recall the subspace $\mathfrak{L}$ of $\mathbb{C}[\eta,\eta^{-1}]$ from \eqref{L}.
In this section, we def\/ine a Hermitian inner product on $\mathfrak{L}$ and show that the non-symmetric dual $q$-Krawtchouk polynomials $\ell^{\pm}_i$ are orthogonal with respect to that inner product.

From Proposition \ref{C;K.v0} and Lemma \ref{giving minimal polynomial} it follows that the minimal polynomial of $\mathcal{X}$ on $\mathbf{W}$ has degree $2D$ and is given by $\eta^Dp^{\perp}h_{D-1}^{\perp}$, which has the following $2D$ simple zeros:
\begin{gather}\label{lambda's}
	\lambda_i:=\tau q^i, \qquad 0\leqslant i \leqslant D-1, \qquad \lambda_{-i}:=\tau^{-1}q^{-i}, \qquad 1\leqslant i\leqslant D.
\end{gather}
In particular, $\mathcal{X}$ is multiplicity-free on $\mathbf{W}$ with the above eigenvalues, and is therefore diagona\-li\-zable on $\mathbf{W}$.
Our aim is to explicitly describe eigenvectors of $\mathcal{X}$ on $\mathbf{W}$. To this end, recall the Leonard systems $\Phi$, $\Phi^{\perp}$ on $M\hat{x}$, $M\hat{x}^{\perp}$ from Propositions~\ref{Phi} and~\ref{para.seq}, respectively, and observe that $\{E_i\hat{x}\}^D_{i=0}$, $\{E_iu^{\perp}_0\}^{D-1}_{i=1}$ form $\Phi^*$- and $(\Phi^{\perp})^*$-standard bases for $M\hat{x}$, $M\hat{x}^{\perp}$, respectively; cf.~\eqref{relatives}. We will work with the following ordered orthogonal basis for $\mathbf{W}$:
\begin{gather*}
	\mathcal{B} = \big\{E_0\hat{x}, E_1\hat{x}, E_1u^{\perp}_0, E_2\hat{x}, E_2u^{\perp}_0,
		\dots, E_{D-1}\hat{x}, E_{D-1}u^{\perp}_0, E_D\hat{x}\big\}.
\end{gather*}

Recall the orthogonal projection $\pi$ (resp.~$\widetilde{\pi}$) from $\mathbf{W}$ onto $M\hat{x}$ (resp.~$M\hat{C}$).

\begin{Lemma}\label{pi, pi.tilde wrt B}
The following $(i)$, $(ii)$ hold:
\begin{enumerate}\itemsep=0pt
\item[$(i)$] The matrix representing $\pi$ in $\mathcal{B}$ is
\begin{gather*}
	\operatorname{blockdiag}\bigl(\pi(0),\pi(1),\dots,\pi(D-1),\pi(D)\bigr),
\end{gather*}
where $\pi(0)=\pi(D)=(1)$, and $\pi(i)=\operatorname{diag}(1,0)$ for $1\leqslant i\leqslant D-1$.
\item[$(ii)$] The matrix representing $\widetilde{\pi}$ in $\mathcal{B}$ is
\begin{gather*}
	\operatorname{blockdiag}\bigl(\widetilde{\pi}(0),\widetilde{\pi}(1),\dots,\widetilde{\pi}(D-1),\widetilde{\pi}(D)\bigr),
\end{gather*}
where $\widetilde{\pi}(0)=(1)$, $\widetilde{\pi}(D)=(0)$, and for $1\leqslant i\leqslant D-1$,
\begin{gather*}
	\widetilde{\pi}(i)=\frac{1}{{(q^D-1)(1+q^e)}}\\
\hphantom{\widetilde{\pi}(i)=}{}\times \begin{pmatrix} (q^e+q^i)(q^{D-i}-1) & q^{-1}(1+q^{D+e-i})(q^i-1)(q^e+q^i)(q^{D-i}-1) \\ q & (1+q^{D+e-i})(q^i-1) \end{pmatrix}.
\end{gather*}
\end{enumerate}
\end{Lemma}

\begin{proof}
(i) Immediate from $\pi E_i\hat{x}=E_i\hat{x}$, $0\leqslant i\leqslant D$, and $\pi E_i u^{\perp}_0=0$, $1\leqslant i\leqslant D-1$.

(ii) We also recall the Leonard systems $\widetilde{\Phi}$, $\widetilde{\Phi}^{\perp}$ on $M\hat{C}$, $M\hat{C}^{\perp}$ from Propositions \ref{LS MC} and~\ref{para.seq.tilde}, respec\-tively, and consider the $\widetilde{\Phi}^*$- and $(\widetilde{\Phi}^{\perp})^*$-standard bases $\{E_i\hat{C}_0\}_{i=0}^{D-1}$, $\{E_i\widetilde{u}_0^{\perp}\}_{i=1}^D$ for \linebreak $M\hat{C}$,~$M\hat{C}^{\perp}$, respectively.
First, we obtain $\widetilde{\pi}(0)=(1)$ and $\widetilde{\pi}(D)=(0)$ since $E_0\hat{x}\in M\hat{C}$ and $E_D\hat{x}\in M\hat{C}^{\perp}$.
Next, let $1\leqslant i\leqslant D-1$. By~\eqref{C;v-com-1},~\eqref{ui.perp.tilde}, and since $\hat{C}_0=\hat{C}_0^-+\hat{C}_0^+$, we have
\begin{gather*}
	\hat{C}_0 = \hat{x}+\frac{q-1}{q^D-1}A\hat{x} + \frac{q}{q^D-1}u^{\perp}_0, \qquad \widetilde{u}_0^{\perp} = -q^e\hat{x}+\frac{q-1}{q^D-1}A\hat{x} + \frac{q}{q^D-1}u^{\perp}_0.
\end{gather*}
Then, using \eqref{e-val} we have
\begin{gather*}
	E_i\hat{C}_0 = \left( 1+\frac{q-1}{q^D-1}\theta_i \right)\!E_i\hat{x}+\frac{q}{q^D-1}E_iu^{\perp}_0 = \frac{(q^e+q^i)(q^{D-i}-1)}{q^D-1}E_i\hat{x}+\frac{q}{q^D-1}E_iu^{\perp}_0,
\end{gather*}
and likewise we have
\begin{gather*}
	E_i\widetilde{u}_0^{\perp}=\frac{(1+q^{D+e-i})(1-q^i)}{q^D-1}E_i\hat{x}+\frac{q}{q^D-1}E_iu^{\perp}_0.
\end{gather*}
Solving these equations for $E_i\hat{x}$, $E_iu^{\perp}_0$, it follows that
\begin{gather*}
	E_i\hat{x} = \frac{1}{1+q^e}E_i\hat{C}_0 -\frac{1}{1+q^e}E_i\widetilde{u}_0^{\perp}, \\
	E_iu^{\perp}_0 = \frac{(1+q^{D+e-i})(q^i-1)}{q(1+q^e)}E_i\hat{C}_0 + \frac{(q^e+q^i)(q^{D-i}-1)}{q(1+q^e)}E_i\widetilde{u}_0^{\perp}.
\end{gather*}
By these comments and since $\widetilde{\pi}E_i\hat{C}_0=E_i\hat{C}_0$, $\widetilde{\pi}E_i\widetilde{u}_0^{\perp}=0$, we obtain $\widetilde{\pi}(i)$ given above.
\end{proof}

\begin{Lemma}\label{Mat_Y}
The matrix representing $\mathcal{X}$ on $\mathbf{W}$ in $\mathcal{B}$ is
\begin{gather*}
	\operatorname{blockdiag}\bigl(x(0),x(1),\dots,x(D-1),x(D)\bigr),
\end{gather*}
where $x(0)=(\tau)$, $x(D)=(\tau^{-1}q^{-D})$, and for $1\leqslant i\leqslant D-1$,
\begin{gather*}
	x(i)=\frac{\tau}{q^D-1}\begin{pmatrix} q^{D+e-i}-q^{D+e}+q^D-q^i & q^{-1}(1+q^{D+e-i})(q^i-1)(q^e+q^i)(q^{D-i}-1) \\ q^{D+1} & -q^{2D+e-i}+q^{D+i}+q^{D+e}-q^D \end{pmatrix}.
\end{gather*}
\end{Lemma}

\begin{proof}
Since $\mathcal{X}=\mathcal{T}'\mathcal{T}$, the result routinely follows from Theorem \ref{thm2} and Lemma \ref{pi, pi.tilde wrt B}.
\end{proof}

Def\/ine the vectors $\mathbf{y}_i$, $-D\leqslant i\leqslant D-1$, by
\begin{gather*}
	\mathbf{y}_i = \frac{(q^{D-i}-1)(1+q^{D+e-i})}{(q^D-1)(1+q^{D+e-2i})} E_i \hat{x} + \frac{q^{D-i+1}}{(q^D-1)(1+q^{D+e-2i})} E_iu^{\perp}_0, \\
	\mathbf{y}_{-i} = \frac{q^{D-2i}(q^i-1)(q^e+q^i)}{(q^D-1)(1+q^{D+e-2i})} E_i \hat{x} - \frac{q^{D-i+1}}{(q^D-1)(1+q^{D+e-2i})} E_iu^{\perp}_0
\end{gather*}
for $1\leqslant i\leqslant D-1$, and
\begin{gather*}
	\mathbf{y}_0=E_0\hat{x}, \qquad \mathbf{y}_{-D}=E_D\hat{x}.
\end{gather*}
Observe that the $\mathbf{y}_i$ are real vectors. We normalized the $\mathbf{y}_i$ so that
\begin{gather}\label{y+y^inv=x}
	\mathbf{y}_i + \mathbf{y}_{-i} = E_i\hat{x}, \qquad 1 \leqslant i \leqslant D-1.
\end{gather}

\begin{Proposition}\label{Eigvec_Y}
For $-D \leqslant i \leqslant D-1$, $\mathbf{y}_i$ is an eigenvector of $\mathcal{X}$ on $\mathbf{W}$ with eigenvalue $\lambda_i$. Moreover, we have $\sum\limits^{D-1}_{i=-D} \mathbf{y}_i = \hat{x}$.
\end{Proposition}

\begin{proof}
The f\/irst assertion follows from direct computations using Lemma \ref{Mat_Y}. By \eqref{y+y^inv=x}, we have
\begin{gather*}
	\sum^{D-1}_{i=-D}\mathbf{y}_i = \mathbf{y}_0 + \sum^{D-1}_{i=1}(\mathbf{y}_i+\mathbf{y}_{-i}) + \mathbf{y}_{-D} = \sum^D_{i=0} E_i\hat{x} = \hat{x},
\end{gather*}
as desired.
\end{proof}

By Proposition \ref{Eigvec_Y}, the $\mathbf{y}_i$ form an eigenbasis of $\mathcal{X}$ on $\mathbf{W}$. Since $\mathcal{B}$ is an orthogonal basis for $\mathbf{W}$, it follows that $\langle\mathbf{y}_i,\mathbf{y}_j\rangle\ne 0$ implies $j\in\{i,-i\}$. We now compute these inner products.

\begin{Lemma}\label{mi values}
The following $(i)$, $(ii)$ hold:
\begin{enumerate}\itemsep=0pt
\item[$(i)$] For $0\leqslant i\leqslant D$, we have
\begin{gather*}
	||E_i\hat{x}||^2 = (-1)^i q^{i(D+e-i+1)} \frac{(q^{-D};q)_i (1+q^{D+e-2i})}{(q;q)_i (-q^{e-i};q)_{D+1}}.
\end{gather*}
\item[$(ii)$] For $1\leqslant i\leqslant D-1$, we have
\begin{gather*}
	||E_iu_0^{\perp}||^2 = (-1)^{i-1} q^{i(D+e-i+1)+D-2} \frac{(q^{-D};q)_{i+1} (1+q^{D+e-2i})}{(q;q)_{i-1} (-q^{e-i+1};q)_{D-1}}.
\end{gather*}
\end{enumerate}
\end{Lemma}

\begin{proof}
Let $m_i$ (resp.~$m_i^{\perp}$) denote the scalar $\mathsf{m}_i$ from \eqref{mi} for the Leonard system $\Phi$ (resp.~$\Phi^{\perp}$).

(i) Since $||E_i\hat{x}||^2=\langle \hat{x},E_0^*E_iE_0^*\hat{x}\rangle=m_i$, the result follows from Proposition~\ref{Phi} and~\eqref{mi formula}.

(ii) By \eqref{ui.perp} and Lemma \ref{|C|;q-exp}, we have
\begin{gather*}
	||u_0^{\perp}||^2 = q^{e-1}\big(q^{D-1}-1\big)\big(q^D-1\big).
\end{gather*}
Since $||E_iu_0^{\perp}||^2=m_{i-1}^{\perp}||u_0^{\perp}||^2$ in this case, the result follows from Proposition \ref{para.seq} and \eqref{mi formula}.
\end{proof}

\begin{Lemma}\label{norms etc}
For $1 \leqslant i \leqslant D-1$, we have
\begin{gather*}
	||\mathbf{y}_i||^2 = (-1)^i q^{i(D+e-i)} \frac{(q^{1-D};q)_i(1+q^D+q^{D+e-i}-q^{D-i})}{(q;q)_i(-q^{e-i};q)_D(1+q^{D+e-2i})}, \\
	||\mathbf{y}_{-i}||^2 = (-1)^{i-1} q^{i(D+e-i-2)+D+e} \frac{(q^{1-D};q)_{i-1}(1+q^D+q^{i-e}-q^i)}{(q;q)_{i-1}(-q^{e-i+1};q)_D(1+q^{D+e-2i})}, \\
	\langle \mathbf{y}_i,\mathbf{y}_{-i} \rangle = (-1)^i q^{i(D+e-i-1)+D} \frac{(q^{1-D};q)_i}{(q;q)_{i-1}(-q^{e-i+1};q)_{D-1}(1+q^{D+e-2i})}.
\end{gather*}
Moreover,
\begin{gather*}
	||\mathbf{y}_0||^2 = \frac{1}{(-q^e;q)_D}, \qquad ||\mathbf{y}_{-D}||^2 = \frac{1}{(-q^{-e};q)_D}.
\end{gather*}
\end{Lemma}

\begin{proof}
The above values are routinely computed using Lemma \ref{mi values}.
\end{proof}

\begin{Proposition}\label{oth.rel H-mod}
For $f,g \in \mathfrak{L}$, we have
\begin{gather*}
	\langle f(\mathcal{X}).\hat{x}, g(\mathcal{X}).\hat{x}\rangle = \sum^{D-1}_{i=-D}\! f(\lambda_i)\overline{g(\lambda_i)}\, ||\mathbf{y}_i||^2 + \sum_{i=1}^{D-1} \bigl(f(\lambda_i)\overline{g(\lambda_{-i})}+f(\lambda_{-i})\overline{g(\lambda_i)}\bigr)\langle \mathbf{y}_i,\mathbf{y}_{-i} \rangle.
\end{gather*}
\end{Proposition}

\begin{proof} From Lemma \ref{Eigvec_Y} it follows that
\begin{gather*}
	f(\mathcal{X}).\hat{x} = f(\mathcal{X}).\!\sum^{D-1}_{i=-D}\!\mathbf{y}_i =\sum^{D-1}_{i=-D}\!f(\lambda_i)\mathbf{y}_i,
\end{gather*}
and similarly for $g(\mathcal{X}).\hat{x}$. Since the $\mathbf{y}_i$ are real, the result is immediate from this comment.
\end{proof}

Def\/ine the Hermitian form $\langle \cdot, \cdot \rangle_{\mathfrak{L}}$ on the subspace $\mathfrak{L}$ from \eqref{L} by\footnote{We would like to point out here that the terms involving the $\langle \mathbf{y}_i,\mathbf{y}_{-i} \rangle$ are missing in the def\/inition of the corresponding Hermitian form for the non-symmetric $q$-Racah polynomials given in \cite[equation~(77)]{JHL2}. A~corrected statement will be given elsewhere.}
\begin{gather}\label{bl_L}
	\langle f,g\rangle_{\mathfrak{L}} = \sum^{D-1}_{i=-D}\! f(\lambda_i)\overline{g(\lambda_i)} ||\mathbf{y}_i||^2 + \sum_{i=1}^{D-1} \bigl(f(\lambda_i)\overline{g(\lambda_{-i})}+f(\lambda_{-i})\overline{g(\lambda_i)}\bigr)\langle \mathbf{y}_i,\mathbf{y}_{-i} \rangle
\end{gather}
for $f,g \in \mathfrak{L}$, where the $\lambda_i$ are from \eqref{lambda's}, and the $||\mathbf{y}_i||^2$ and the $\langle \mathbf{y}_i,\mathbf{y}_{-i} \rangle$ are given in Lemma~\ref{norms etc}. Since $\langle f,f\rangle_{\mathfrak{L}}=0$ implies $f\in \bigl(p^{\perp}h_{D-1}^{\perp}\mathbb{C}[\eta,\eta^{-1}]\bigr)\cap\mathfrak{L}=0$, it follows that $\langle \cdot, \cdot \rangle_{\mathfrak{L}}$ is an inner product on $\mathfrak{L}$. From \eqref{y+y^inv=x} it also follows that, when restricted to symmetric polynomials, \eqref{bl_L} gives an inner product for which the (symmetric) dual $q$-Krawtchouk polynomials \eqref{dual q-Krawtchouk} for the Leonard system $\Phi$ are orthogonal; cf.~\cite[Section~14.17]{KLS2010B}, \cite[Section~16]{Terwilliger2004LAA}.

We now present the orthogonality relation for the non-symmetric dual $q$-Krawtchouk polynomials $\ell_i^{\pm}$ from Def\/inition \ref{nDK}:

\begin{Theorem}\label{thm3}
With reference to the inner product \eqref{bl_L}, we have
\begin{gather*}
	\langle \ell^{\mu}_i, \ell^{\nu}_j\rangle_{\mathfrak{L}} = \delta_{i,j}\delta_{\mu,\nu}\, |C_i^{\mu}|
\end{gather*}
for $0\leqslant i,j\leqslant D-1$ and $\mu,\nu\in \{+,-\}$, where the $|C_i^{\pm}|$ are given in Lemma~{\rm \ref{|C|;q-exp}}.
\end{Theorem}

\begin{proof}Immediate from Propositions \ref{C;K.v0} and \ref{oth.rel H-mod}.
\end{proof}

\subsection*{Acknowledgements}

The authors thank Daniel Orr for helpful comments on nil-DAHAs of type $(C_n^{\vee},C_n)$. They also thank Paul Terwilliger for many valuable discussions, Marta Mazzocco and Alexei Zhedanov for bringing the authors' attention to \cite{Mazzocco2014SIGMA,Mazzocco2016N}, and the anonymous referees for carefully reading the paper. Part of this work was done while Jae-Ho Lee was visiting Tohoku University as a~JSPS Postdoctoral Fellow. Hajime Tanaka was supported by JSPS KAKENHI Grant Numbers JP25400034 and JP17K05156. An extended abstract of this work appeared in the proceedings of FPSAC'17, London, UK, July 2017; cf.~\cite{FPSAC}.

\pdfbookmark[1]{References}{ref}
\LastPageEnding

\end{document}